\documentclass[11 pt]{article}
\title{The high-dimensional cohomology of the moduli space of curves with level structures}
\author{Neil J. Fullarton\thanks{Supported in part by an AMS-Simons grant and an MSRI postdoctoral fellowship}\ \ and Andrew Putman\thanks{Supported in part by NSF grant DMS-1255350 and the Alfred P.\ Sloan Foundation}}%\vspace{-8pt}}
\date{November 29, 2017}

\usepackage{amsmath,amssymb,amsthm,amscd,amsfonts}
\usepackage{epsfig,pinlabel,paralist}
\usepackage[hmargin=1.25in, vmargin=1in]{geometry}
\usepackage[font=small,format=plain,labelfont=bf,up,textfont=it,up]{caption}
\usepackage{type1cm}
\usepackage{calc}
\usepackage{enumerate}
\usepackage{url}
\usepackage{bm}
\usepackage{booktabs}
\usepackage{microtype}
\usepackage{titling}
\usepackage[all,cmtip]{xy}

\DeclareFontFamily{U}{matha}{}
\DeclareFontShape{U}{matha}{m}{n}{
  <-5.5>    matha5
  <5.5-6.5> matha6 
  <6.5-7.5> matha7
  <7.5-8.5> matha8
  <8.5-9.5> matha9
  <9.5-11>  matha10
  <11->     matha12
}{}
\DeclareSymbolFont{matha}{U}{matha}{m}{n}
\DeclareFontSubstitution{U}{matha}{m}{n}
\DeclareFontFamily{U}{mathx}{\hyphenchar\font45}
\DeclareFontShape{U}{mathx}{m}{n}{<-> mathx10}{}
\DeclareSymbolFont{mathx}{U}{mathx}{m}{n}
\DeclareFontSubstitution{U}{mathx}{m}{n}

\DeclareMathDelimiter{\ldbrack}{4}{matha}{"76}{mathx}{"30}
\DeclareMathDelimiter{\rdbrack}{5}{matha}{"77}{mathx}{"38}

\newcommand{\coloneq}{\mathrel{\resizebox{\widthof{$\mathord{=}$}}{\height}{ $\!\!\resizebox{1.2\width}{0.8\height}{\raisebox{0.23ex}{$\mathop{:}$}}\!\!=\!\!$ }}}

\usepackage[dvipdfm, bookmarks, bookmarksdepth=2, colorlinks=true, linkcolor=blue, citecolor=blue, urlcolor=blue]{hyperref}

\theoremstyle{plain}
\newtheorem{theorem}{Theorem}[section]
\newtheorem{maintheorem}{Theorem}
\newtheorem{maintheoremprime}{Theorem}
\newtheorem{proposition}[theorem]{Proposition}
\newtheorem{lemma}[theorem]{Lemma}

\newtheorem{conjecture}[theorem]{Conjecture}

\newtheorem*{claim}{Claim}

\newtheorem{stepa}{Step}

\theoremstyle{definition}

\newtheorem{definition}[theorem]{Definition}

\theoremstyle{remark}
\newtheorem{remark}[theorem]{Remark}

\DeclareMathOperator{\Coker}{coker}

\DeclareMathOperator{\Mod}{Mod}

\DeclareMathOperator{\Sp}{Sp}

\DeclareMathOperator{\SL}{SL}

\newcommand\Moduli{\ensuremath{{\mathcal M}}}

\newcommand\C{\ensuremath{\mathbb{C}}}
\newcommand\Z{\ensuremath{\mathbb{Z}}}
\newcommand\Q{\ensuremath{\mathbb{Q}}}

\newcommand\Field{\ensuremath{\mathbb{F}}}

\DeclareMathOperator{\HH}{H}
\newcommand\RH{\ensuremath{\widetilde{\HH}}}
\DeclareMathOperator{\CC}{C}
\newcommand\RC{\ensuremath{\widetilde{\CC}}}
\DeclareMathOperator{\ZZ}{Z}
\newcommand\RZ{\ensuremath{\widetilde{\ZZ}}}

\DeclareMathOperator{\Interior}{Int}
\newcommand\Span[1]{\ensuremath{\langle #1 \rangle}}

\newcommand\Set[2]{\ensuremath{\{\text{#1 $|$ #2}\}}}

\newcommand\Figure[4]{
\begin{figure}[t]
\centering
\centerline{\psfig{file=#2,scale=#4}}
\caption{#3}
\label{#1}
\end{figure}}

\DeclareMathOperator{\CohCD}{CohCD}
\DeclareMathOperator{\St}{St}
\DeclareMathOperator{\Ind}{Ind}
\DeclareMathOperator{\ns}{ns}
\DeclareMathOperator{\sep}{sep}
\DeclareMathOperator{\id}{id}
\newcommand\StNS{\ensuremath{\St^{\ns}}}
\newcommand\StSep{\ensuremath{\St^{\sep}}}
\newcommand\Curves{\ensuremath{\mathcal{C}}}
\newcommand\Arc{\ensuremath{\mathcal{AC}}}
\newcommand\Tits{\ensuremath{\mathcal{T}}}
\newcommand\Apartment{\ensuremath{\mathfrak{A}}}
\newcommand\fX{\ensuremath{\mathfrak{X}}}
\newcommand\fS{\ensuremath{\mathfrak{S}}}
\newcommand\hS{\ensuremath{\widehat{S}}}
\newcommand\tphi{\ensuremath{\widetilde{\phi}}}

\newcommand\tzeta{\ensuremath{\widetilde{\zeta}}}
\newcommand\oB{\ensuremath{\overline{B}}}
\newcommand\oT{\ensuremath{\overline{T}}}
\newcommand\ox{\ensuremath{\overline{x}}}
\newcommand\NS[1]{\ensuremath{\ldbrack #1 \rdbrack}}
\newcommand\inc{\ensuremath{\iota}}

\newcommand\proj{\ensuremath{\pi}}
\newcommand\chainproj{\ensuremath{\rho}}

\DeclareMathOperator{\Sym}{Sym}
\newcommand\cSym{\ensuremath{\mathcal{SYM}}}
\newcommand\hSym{\ensuremath{\widehat{\Sym}}}
\DeclareMathOperator{\indec}{indec}
\DeclareMathOperator{\dec}{dec}
\newcommand\cP{\ensuremath{\mathcal{P}}}
\newcommand\cO{\ensuremath{\mathcal{O}}}
\newcommand\bV{\ensuremath{\mathbf{V}}}
\newcommand\bVdec{\ensuremath{\mathbf{V}^{\dec}}}
\newcommand\bVindec{\ensuremath{\mathbf{V}^{\indec}}}
\newcommand\piindec{\ensuremath{\pi^{\indec}}}
\newcommand\cF{\ensuremath{\mathcal{F}}}

\begin{document}

\maketitle

%\vspace{-30pt}
\begin{abstract}
We prove that the moduli space of curves with level structures has an enormous amount of rational cohomology
in its cohomological dimension.  As an application, we prove that the coherent cohomological dimension of the
moduli space of curves is at least $g-2$.  Well known conjectures of Looijenga would imply that this is sharp.
\end{abstract}

\section{Introduction}

Let $\Sigma_g$ be a closed oriented genus $g$ surface.  The {\em mapping class group} of $\Sigma_g$, denoted
$\Mod_g$, is the group of isotopy classes of orientation-preserving diffeomorphisms of $\Sigma_g$.  The group
$\Mod_g$ lies at the crossroads of many areas of mathematics.  One fundamental reason for this is that $\Mod_g$
is the orbifold fundamental group of the moduli space $\Moduli_g$ of genus $g$ Riemann surfaces.  In fact,
even more is true: as an orbifold, $\Moduli_g$ is an Eilenberg--MacLane space for $\Mod_g$, which implies in
particular that
\[\HH^{\ast}(\Mod_g;\Q) \cong \HH^{\ast}(\Moduli_g;\Q).\]
See \cite{FarbMargalitPrimer} for a survey of $\Mod_g$ and $\Moduli_g$.

\paragraph{Stable cohomology.}
Let $\kappa_i \in \HH^{2i}(\Mod_g;\Q)$ be the $i^{\text{th}}$ Miller--Morita--Mumford class.  We then
have a graded ring homomorphism $\Q[\kappa_1,\kappa_2,\ldots] \rightarrow \HH^{\ast}(\Mod_g;\Q)$, and
the Mumford conjecture (proved by Madsen--Weiss \cite{MadsenWeiss}) says that this graded ring homomorphism is
an isomorphism in degrees less than or equal to $\frac{2}{3}(g-1)$.  Aside from some low-genus computations,
no nontrivial elements of $\HH^{\ast}(\Mod_g;\Q)$ have been found outside this stable range.  However,
Harer--Zagier \cite{HarerZagierEuler} proved that the Euler characteristic of $\Mod_g$ is enormous, so there
must exist vast amounts of unstable rational cohomology.

\paragraph{Level structures.}
The rational cohomology of finite-index subgroups of $\Mod_g$ (or, equivalently, finite covers of $\Moduli_g$) is also
of interest.  For $\ell \geq 2$, the {\em level $\ell$ congruence subgroup} of
$\Mod_g$, denoted $\Mod_g(\ell)$, is the kernel of the action of $\Mod_g$ on $\HH_1(\Sigma_g;\Z/\ell)$.  It
fits into a short exact sequence
\[1 \longrightarrow \Mod_g(\ell) \longrightarrow \Mod_g \longrightarrow \Sp_{2g}(\Z/\ell) \longrightarrow 1.\]
The symplectic group appears here because the action of $\Mod_g$ 
on $\HH_1(\Sigma_g;\Z/\ell)$ preserves the algebraic intersection pairing.  The associated finite cover of
$\Moduli_g$ is the moduli space $\Moduli_g(\ell)$ of genus $g$ curves equipped with a full level $\ell$ structure (i.e.\ 
a basis for the $\ell$-torsion in their Jacobian).  
A conjecture of the second
author (see \cite[\S 1]{PutmanH2Level} for a discussion) asserts that
\[\HH_k(\Mod_g(\ell);\Q) \cong \HH_k(\Mod_g;\Q) \quad \quad (g \gg k).\]
This holds for $k=1$ by work of Hain \cite{HainH1Level} and for $k=2$ by work 
of the second author \cite{PutmanH2Level}.

\paragraph{Cohomological dimension.}
The main topic of this paper is what happens outside the stable range.
Harer \cite{HarerDuality} proved that the virtual cohomological dimension (vcd) of $\Mod_g$ is $4g-5$, so 
$\HH^i(\Mod_g;\Q) = 0$ for $i>4g-5$.  The first place where one might hope to find some unstable rational cohomology
is thus in degree $4g-5$.  However, a theorem proved independently by 
Morita--Sakasai--Suzuki \cite{MoritaSakasaiSuzuki} and by Church--Farb--Putman \cite{ChurchFarbPutmanVanish} says
that $\HH^{4g-5}(\Mod_g;\Q) = 0$.  This might seem to contradict the fact that the vcd of $\Mod_g$ is $4g-5$.  However,
the definition of the vcd of a group makes use not only of ordinary cohomology, but also cohomology with respect
to arbitrary twisted coefficient systems.  Harer's theorem thus only asserts that there exists some $\Q[\Mod_g]$-module $M$ (necessarily nontrivial, in light of
\cite{MoritaSakasaiSuzuki,ChurchFarbPutmanVanish}) such that $\HH^{4g-5}(\Mod_g;M) \neq 0$.  

\paragraph{Main theorem.}
This brings us to our main theorem, which says that in contrast to what conjecturally happens in the stable range,
the group $\Mod_g(\ell)$ has an enormous amount of rational cohomology in its vcd.

\begin{maintheorem}
\label{maintheorem:cohomology}
Fix $g, \ell \geq 2$.  Let $p$ be a prime dividing $\ell$.  Then
\[\dim_{\Q} \HH^{4g-5}(\Mod_g(\ell);\Q) \geq \frac{|\Sp_{2g}(\Field_p)|}{g(p^{2g}-1)} = \frac{1}{g} p^{2g-1} \prod_{k=1}^{g-1} (p^{2k}-1)p^{2k-1}.\] 
\end{maintheorem}

\begin{remark}
Our lower bound is super-exponential in $g$; its leading term is $\frac{1}{g} p^{\binom{2g}{2}}$.  To give an idea of how
quickly it grows, the following are some special cases:

\centerline{\begin{tabular}{@{}r@{\hspace{0.05in}}c@{\hspace{0.05in}}l@{\hspace{0.5in}}r@{\hspace{0.05in}}c@{\hspace{0.05in}}l@{}}
$\dim_{\Q} \HH^{3}(\Mod_2(2);\Q)$ & $\geq$ & $24$ & $\dim_{\Q} \HH^{3}(\Mod_2(3);\Q)$ & $\geq$ & $216$ \\
$\dim_{\Q} \HH^{7}(\Mod_3(2);\Q)$ & $\geq$ & $11520$ & $\dim_{\Q} \HH^{7}(\Mod_3(3);\Q)$ & $\geq$ & $4199040$ \\
$\dim_{\Q} \HH^{11}(\Mod_4(2);\Q)$ & $\geq$ & $92897280$ & $\dim_{\Q} \HH^{11}(\Mod_4(3);\Q)$ & $\geq$ & $6685442749440$.
\end{tabular}}
\end{remark}

\begin{remark}
Our lower bound is almost certainly not sharp.
One difficulty with improving it is that as described below,
we exploit a connection to the Tits building for the group $\SL_n(\Field_{p})$.  Presumably better results
could be obtained by studying the obvious analogue of this building for the finite group $\SL_n(\Z/\ell)$, but
for $\ell$ not prime this is poorly understood.  We also expect that it is not sharp for $\ell$ a prime, though
we do not have any concrete proposals for improving it in that case.
\end{remark}

\begin{remark}
In his 1986 paper \cite{HarerDuality} (see p.\ 175), Harer asserts that ``it is possible to show'' that
$\HH^{4g-5}(\Mod_g(\ell);\Q) \neq 0$.  However, he never published a proof of this.
\end{remark}

\paragraph{Application to algebraic geometry.}
Theorem \ref{maintheorem:cohomology} has an interesting application to the geometry of the (coarse) moduli space $\Moduli_g$
of genus $g$ Riemann surfaces, which is a quasiprojective complex variety.  
We begin with the following conjecture of Looijenga \cite{FaberLooijenga}.

\begin{conjecture}[Looijenga]
\label{conjecture:looijenga}
For $g \geq 2$, the variety $\Moduli_g$ can be covered by $(g-1)$ open affine subsets.
\end{conjecture}

For example, this conjecture asserts that $\Moduli_2$ is itself affine, which is a consequence of the fact
that every genus $2$ Riemann surface is hyperelliptic.  More generally, Fontanari--Pascolutti \cite{FontanariPascolutti}
proved that Conjecture \ref{conjecture:looijenga} holds for $2 \leq g \leq 5$.

Conjecture \ref{conjecture:looijenga} would imply a bound on the coherent cohomological dimension of $\Moduli_g$, which
is defined as follows.
If $X$ is a variety, then the {\em coherent cohomological dimension} of $X$, denoted $\CohCD(X)$, is the maximum
value of $k$ such that there exists some coherent sheaf $\cF$ on $X$ with $\HH^k(X;\cF) \neq 0$.
The coherent cohomological dimension of a variety reflects interesting geometric properties
of the variety.  For example, Serre \cite{SerreAffine}
proved that $\CohCD(X) = 0$ if and only
if $X$ is an affine variety.  See \cite{HartshorneDimension} for more information on
coherent cohomological dimension.

Since $\Moduli_g$ is separated, the intersection of two affine open subsets of $\Moduli_g$ is itself an affine
open subset.  Thus if Conjecture \ref{conjecture:looijenga} were true and $\cF$ were a coherent sheaf
on $\Moduli_g$, then we could apply the Mayer--Vietoris spectral sequence to the cover of $\Moduli_g$ given by
Conjecture \ref{conjecture:looijenga} to deduce that $\HH^k(\Moduli_g;\cF) = 0$ for $k>g-2$.  In other
words, Conjecture \ref{conjecture:looijenga} would imply that $\CohCD(\Moduli) \leq g-2$.

Using our main theorem (Theorem \ref{maintheorem:cohomology}), we will prove the following, which asserts that this conjectural upper bound is sharp.

\begin{maintheorem}
\label{maintheorem:cohcd}
For $g \geq 2$, we have $\CohCD(\Moduli_g) \geq g-2$ with equality for $2 \leq g \leq 5$.
\end{maintheorem}

\begin{remark}
This implies that $\Moduli_g$ cannot be covered with fewer than $(g-1)$ open affine subsets.  This
was already known.  Indeed, it follows from work of Chaudhuri \cite{Chaudhuri}, who gave a lower
bound on the cohomological excess of $\Moduli_g$ (which is defined using constructible sheaves).
\end{remark}

\begin{remark}
The only paper we are aware of concerning upper bounds for things related to $\CohCD(\Moduli_g)$ is recent work of
Mondello \cite{Mondello}, who proved that the Dolbeault cohomological dimension
of $\Moduli_g$ is at most $2g-2$.  The Dolbeault cohomological dimension of a complex
analytic variety $X$ is the maximal $k$ such that there exists a holomorphic vector
bundle $\mathcal{B}$ on $X$ such that $\HH^k(\Moduli_g;\mathcal{B}) \neq 0$.  As will
be clear from our argument below, our work also establishes a lower bound of $g-2$ on
the Dolbeault cohomological dimension of $\Moduli_g$.
\end{remark}

\paragraph{Proof of Theorem \ref{maintheorem:cohcd}.}
The derivation of Theorem \ref{maintheorem:cohcd} from Theorem \ref{maintheorem:cohomology}
is so simple that we give it here.  We wish to thank Eduard Looijenga for explaining
this argument to us.  Fix some $\ell \geq 3$, so $\Moduli_g(\ell)$ is
smooth.  The projection $\Moduli_g(\ell) \rightarrow \Moduli_g$ is a finite surjective
map, so $\CohCD(\Moduli_g(\ell)) = \CohCD(\Moduli_g)$ (see
\cite[Proposition 1.1]{HartshorneDimension}).  It is thus enough to prove that
$\CohCD(\Moduli_g(\ell)) \geq g-2$.  Assume for the sake of contradiction that this
is false.  The Hodge--de Rham spectral sequence for $\Moduli_g(\ell)$ converges to
$\HH^{\ast}(\Moduli_g(\ell);\C)$ and has
\[E_1^{pq} = \HH^p(\Moduli_g(\ell);\Omega^q).\]
Since the complex dimension of $\Moduli_g(\ell)$ is $3g-3$, we have $\Omega^q = 0$ for
$q \geq 3g-2$, and thus
\begin{equation}
\label{eqn:e11}
E_1^{pq} = 0 \quad \quad (q \geq 3g-2).
\end{equation}
Moreover, since $\Omega^q$ is a coherent sheaf on $\Moduli_g(\ell)$ our
assumption that $\CohCD(\Moduli_g(\ell)) < g-2$ implies that
\begin{equation}
\label{eqn:e12}
E_1^{pq} = \HH^p(\Moduli_g(\ell);\Omega^q) = 0 \quad \quad (p \geq g-2).
\end{equation}
From \eqref{eqn:e11} and \eqref{eqn:e12},
we deduce that $E_1^{pq} = 0$ whenever $p+q = 4g-5$.  This implies that
$\HH^{4g-5}(\Moduli_g(\ell);\C) = 0$, contradicting Theorem \ref{maintheorem:cohomology}.
The fact that we have equality for $2 \leq g \leq 5$ follows from
the aforementioned theorem of Fontanari--Pascolutti \cite{FontanariPascolutti}
asserting that Conjecture \ref{conjecture:looijenga} holds for $2 \leq g \leq 5$.

\paragraph{Proof outline for Theorem \ref{maintheorem:cohomology}.}
Our proof of Theorem \ref{maintheorem:cohomology} has four steps.
\begin{compactenum}
\item First, we use the fact that the mapping class
group satisfies Bieri--Eckmann duality \cite{HarerDuality} to translate the theorem into an
assertion about the action of $\Mod_g(\ell)$ on the Steinberg module for the mapping class
group, i.e.\ the unique nonzero homology group of the curve complex.  
\item Next, we study this action
by constructing a novel surjective homomorphism from the Steinberg module for the mapping class group
to a vector space $\StNS_{2g}(\Field_p)$ 
that is a quotient of the Steinberg module $\St_{2g}(\Field_p)$ for the finite group $\SL_{2g}(\Field_{p})$, i.e.\ the
unique nonzero homology group of this finite group's Tits building.  
\end{compactenum}
It follows from the previous two steps that the dimension of $\HH^{4g-5}(\Mod_g(\ell);\Q)$ is at least
$\dim_{\Q} \StNS_{2g}(\Field_p)$.  
At this point, one might think that we have at least proved that $\HH^{4g-5}(\Mod_g(\ell);\Q) \neq 0$.  However,
there is a problem -- from its definition, it is not clear that $\StNS_{2g}(\Field_p) \neq 0$.  The next two
steps analyze this vector space.
\begin{compactenum}
\item[3.] The third step is representation-theoretic.  As a prelude to analyzing $\StNS_{2g}(\Field_p)$,
we show how to decompose the restriction of the
$\SL_{2g}(\Field_{p})$-representation $\St_{2g}(\Field_p)$ to the subgroup $\Sp_{2g}(\Field_{p})$, i.e.\ we construct a
branching rule between these two different classical groups (see Remark \ref{remark:branching} below).
\item[4.] Finally, to use the third step to show that $\StNS_{2g}(\Field_p)$ is 
$\frac{|\Sp_{2g}(\Field_p)|}{g(p^{2g}-1)}$-dimensional, we apply
classical theorems concerning the partition function and exponential generating functions (some of which
go back to Euler).  
\end{compactenum}

\begin{remark}
One might expect that the Steinberg module for the finite symplectic group $\Sp_{2g}(\Field_{p})$ would appear
here rather than the Steinberg module for $\SL_{2g}(\Field_{p})$.  We tried to do
this initially, but were unsuccessful.  Indeed, every map from the Steinberg module for the
mapping class group to the Steinberg module for $\Sp_{2g}(\Field_p)$ that we were able to concoct ended up
being the zero map.
What is more, our lower bound on the dimension of $\HH^{4g-5}(\Mod_g(\ell);\Q)$ is significantly larger than the dimension
of the Steinberg module for $\Sp_{2g}(\Field_p)$, namely $p^{g^2}$, so it seems unlikely that one could 
use the Steinberg module for $\Sp_{2g}(\Field_p)$ to prove a theorem as strong as Theorem \ref{maintheorem:cohomology}.
\end{remark}

\paragraph{Outline.}
Our proof of Theorem \ref{maintheorem:cohomology} is contained in \S \ref{section:duality} -- \ref{section:bounds}.
The first part of the proof is in \S \ref{section:duality}, which as described above uses Bieri--Eckmann
duality to connect Theorem \ref{maintheorem:cohomology} to the Steinberg module for the mapping class group.
In \S \ref{section:maptobuilding}, we describe how to connect the Steinberg module for the mapping class group
to the Steinberg module for the special linear group.  Next, in \S \ref{section:decompose} we construct
the branching rule discussed above.  Finally, in \S \ref{section:bounds} we establish the lower bound in
Theorem \ref{maintheorem:cohomology}.  

\paragraph{Acknowledgments.}
We wish to thank Eduard Looijenga for asking us whether Theorem \ref{maintheorem:cohomology} held and explaining
why it implied Theorem \ref{maintheorem:cohcd}.  We also want to thank Thomas Church, Benson Farb, Martin Kassabov, Dan Margalit,
John Wiltshire-Gordon, and the MathOverflow user ``Lucia'' for helpful comments. We finally want to thank the referee
for a careful and thoughtful report.

\section{Bieri--Eckmann duality for the mapping class group}
\label{section:duality}

In this section, we translate Theorem \ref{maintheorem:cohomology} into a statement (Theorem \ref{maintheoremprime:cohomology} below) concerning the Steinberg module for the mapping class group.

Fix some $g \geq 2$.  Ignoring orbifold issues, the space
$\Moduli_g$ is a smooth orientable manifold.  If it were compact, then it would 
satisfy Poincar\'{e} duality.  Unfortunately, it is not compact.  It therefore only satisfies Poincar\'{e}--Lefschetz
duality with a contribution coming from the ``boundary'' (i.e.\ at infinity).  This
is best described in terms of the group $\Mod_g$.  Harer \cite{HarerDuality} proved that
$\Mod_g$ satisfies a version of Poincar\'{e}--Lefschetz duality called virtual Bieri--Eckmann duality.  Recalling that
the vcd of $\Mod_g$ is $4g-5$, this duality takes the form
\begin{equation}
\label{eqn:bierieckmann1}
\HH^{4g-5-i}(\Mod_g;\Q) \cong \HH_i(\Mod_g;\St(\Sigma_g)) \quad \quad (i \geq 0).
\end{equation}
Here the dualizing module $\St(\Sigma_g)$ is the Steinberg module for $\Mod_g$, which is defined as follows.  The {\em curve complex}
for $\Sigma_g$, denoted $\Curves_g$, is the simplicial complex whose $k$-simplices are collections 
$\{\gamma_0,\ldots,\gamma_k\}$ of distinct isotopy classes of non-nullhomotopic simple closed curves on $\Sigma_g$
that can be realized
disjointly.  One way of thinking of $\Curves_g$ is that it is a combinatorial model for all the ways that a genus
$g$ Riemann surface can degenerate to a stable nodal surface.  The group $\Mod_g$ acts on $\Curves_g$, and
Harer \cite{HarerDuality} proved that $\Curves_g$ is homotopy equivalent to an infinite wedge of $(2g-2)$-dimensional
spheres.  The {\em Steinberg module} for $\Mod_g$ is defined to be
\[\St(\Sigma_g) \coloneq \RH_{2g-2}(\Curves_g;\Q).\]  
In the duality \eqref{eqn:bierieckmann1}, this measures the contribution ``at
infinity'' to cohomology.

Since $\Mod_g(\ell)$ is a finite-index subgroup of $\Mod_g$, it also satisfies Bieri--Eckmann duality with the same
dualizing module $\St(\Sigma_g)$.  We thus have
\[\HH^{4g-5-i}(\Mod_g(\ell);\Q) \cong \HH_i(\Mod_g(\ell);\St(\Sigma_g)) \quad \quad (i \geq 0).\]
Recall that if $G$ is a group and $M$ is a $G$-module, then $\HH_0(G;M) \cong M_G$.  Here $M_G$ denotes the 
{\em coinvariants} of $G$ acting on $M$, i.e.\ the quotient of $M$ by the submodule spanned by
$\Set{$m-g(m)$}{$m \in M$, $g \in G$}$.  To prove Theorem \ref{maintheorem:cohomology}, we must prove that
\[\HH_0(\Mod_g(\ell);\St(\Sigma_g)) \cong (\St(\Sigma_g))_{\Mod_g(\ell)}\]
is large, i.e.\ we must construct $\Mod_g(\ell)$-invariant homomorphisms from $\St(\Sigma_g)$ whose targets
are large.  More precisely, Theorem \ref{maintheorem:cohomology} is equivalent to the following theorem.

\begin{maintheoremprime}
\label{maintheoremprime:cohomology}
Fix $g, \ell \geq 2$.  Let $p$ be a prime dividing $\ell$.  Then there 
exists a $\Mod_g(\ell)$-invariant surjective homomorphism
$\psi\colon \St(\Sigma_g) \rightarrow V$ such that $V$ is a vector space over $\Q$ satisfying
\[\dim_{\Q} V = \frac{|\Sp_{2g}(\Field_p)|}{g(p^{2g}-1)}.\]
\end{maintheoremprime}

\section{Mapping to the nonseparated building}
\label{section:maptobuilding}

In this section, we construct the homomorphism $\psi$ whose existence is asserted by Theorem
\ref{maintheoremprime:cohomology} and prove that $\psi$ is surjective.  
Later sections will calculate the dimension of its target.
Though we will not use any of his results explicitly, 
our point of view in this section is heavily influenced by Broaddus's paper \cite{BroaddusSteinberg}.  To avoid
degenerate situations, we will fix some $g \geq 2$.

\paragraph{Arcs.}
To construct $\psi$, we need to understand generators for 
$\St(\Sigma_g) = \RH_{2g-2}(\Curves_g;\Q)$.
For this, we will use an alternate model for $\Curves_g$ that
was constructed by Harer.  Fix a basepoint $\ast \in \Sigma_g$.  The {\em arc complex} on
$\Sigma_g$, denoted $\Arc_g^1$, is the simplicial complex whose $k$-simplices are collections
$\{\alpha_0,\ldots,\alpha_k\}$ of distinct homotopy classes of unoriented $\ast$-based simple closed curves on $\Sigma_g$
that can be realized to be disjoint away from $\ast$.  The superscript $1$ in $\Arc_g^1$ indicates that there is a
single basepoint.  Harer \cite[Theorem 1.5]{HarerStability} proved that $\Arc_g^1$ is contractible.  See
\cite{HatcherTriangulations} for a beautiful short proof of this.  Next, define
$\Arc_g^1(\infty)$ to be the subcomplex of $\Arc_g^1$ consisting of simplices $\{\alpha_0,\ldots,\alpha_k\}$ such that
some component of $\Sigma_g$ cut along $\alpha_0 \cup \cdots \cup \alpha_k$ is not simply-connected.
Harer \cite[Lemma 1.7]{HarerStability} proved that $\Arc_g^1(\infty)$ is homotopy equivalent to $\Curves_g$.

\begin{remark}
\label{remark:oncepunctured}
What Harer actually proves in the above reference is that $\Arc_g^1(\infty)$ is homotopy equivalent to the curve
complex on a once-punctured surface; however, another theorem of Harer from \cite{HarerDuality} is that the curve
complex on a once-punctured surface is homotopy equivalent to the curve complex on a closed surface.  See
\cite{KentLeiningerSchleimer} and \cite{HatcherVogtmannTethers} for alternate proofs of this.
\end{remark}

\begin{lemma}
\label{lemma:relativehomology}
For $g \geq 2$, we have $\St(\Sigma_g) \cong \HH_{2g-1}(\Arc_g^1,\Arc_g^1(\infty);\Q)$.
\end{lemma}
\begin{proof}
Since $\Arc_g^1$ is contractible, the long exact sequence in relative homology implies that
$\HH_i(\Arc_g^1(\infty);\Q) \cong \HH_{i+1}(\Arc_g^1,\Arc_g^1(\infty);\Q)$.  The lemma now follows from 
the fact that $\Arc_g^1(\infty)$ is homotopy equivalent to $\Curves_g$.
\end{proof}

\begin{remark}
The description of $\St(\Sigma_g)$ in Lemma \ref{lemma:relativehomology} slightly obscures
the $\Mod_g$-action, which goes as follows (cf.\ Remark \ref{remark:oncepunctured}).  Consider a mapping class which is
represented by an orientation-preserving diffeomorphism 
$f\colon \Sigma_g \rightarrow \Sigma_g$.  Isotoping $f$, we can assume that
$f(\ast) = \ast$, and thus $f$ induces an automorphism of $\Arc_g^1$ that
preserves $\Arc_g^1(\infty)$.  Under the isomorphism in 
Lemma \ref{lemma:relativehomology}, the induced action of $f$ on
$\HH_{2g-1}(\Arc_g^1,\Arc_g^1(\infty);\Q)$ is precisely the action of the mapping class
associated to $f$ on $\St(\Sigma_g)$.  In particular, this action on homology is
independent of all of our choices.  
\end{remark}

\begin{lemma}
\label{lemma:recognizespan}
Let $R$ be a ring.  For some $g \geq 2$, 
let $\sigma = \{\alpha_0,\ldots,\alpha_k\}$ be a simplex of $\Arc_g^1$.  
Orient each $\alpha_i$ in an arbitrary way.  Then $\sigma$ is a simplex of $\Arc_g^1(\infty)$
if and only if the homology classes of the $\alpha_i$ do not span $\HH_1(\Sigma_g;R)$.
\end{lemma}
\begin{proof}
Let $X$ be a closed regular neighborhood of $\alpha_0 \cup \cdots \cup \alpha_k$ and let 
$Y$ be the result of gluing all simply-connected components of 
$\Sigma_g \setminus \Interior(X)$ to $X$.  By definition, $\sigma$ is a simplex of
$\Arc_g^1(\infty)$ if and only if $Y$ is a proper subsurface of $\Sigma_g$.  Since $X$
deformation retracts to $\alpha_0 \cup \cdots \cup \alpha_k$, the homology classes 
of the $\alpha_i$ span $\HH_1(X;R)$.  Mayer--Vietoris implies that the map 
$\HH_1(X;R) \rightarrow \HH_1(Y;R)$
is surjective, so the homology classes of the $\alpha_i$ also span $\HH_1(Y;R)$.  
If $Y = \Sigma_g$, then
we deduce that the homology classes of the $\alpha_i$ span $\HH_1(\Sigma_g;R)$.  
If $Y$ instead is a proper
subsurface of $\Sigma_g$, then the map $\HH_1(Y;R) \rightarrow \HH_1(\Sigma_g;R)$ 
is not surjective,
so the homology classes of the $\alpha_i$ do not span $\HH_1(\Sigma_g;R)$.  The lemma follows.
\end{proof}

\begin{lemma}
\label{lemma:skeleton}
For $g \geq 2$, the $(2g-2)$-skeletons of $\Arc_g^1$ and $\Arc_g^1(\infty)$ are equal.
\end{lemma}
\begin{proof}
Consider a $k$-simplex $\sigma = \{\alpha_0,\ldots,\alpha_k\}$ of $\Arc_g^1$ that does not lie in $\Arc_g^1(\infty)$.  Fix an orientation on each $\alpha_i$.  
Lemma \ref{lemma:recognizespan} implies that
the homology classes of the $\alpha_i$ span $\HH_1(\Sigma_g;\Z)$.  This implies that $k+1 \geq 2g$, as desired.
\end{proof}

\begin{lemma}
\label{lemma:arcdescription}
For $g \geq 2$, we have $\St(\Sigma_g) \cong \Coker(\CC_{2g}(\Arc_g^1;\Q) \stackrel{\partial}{\rightarrow} \CC_{2g-1}(\Arc_g^1,\Arc_g^1(\infty);\Q))$.
\end{lemma}
\begin{proof}
Immediate from Lemmas \ref{lemma:relativehomology} and \ref{lemma:skeleton}.
\end{proof}

\paragraph{The building.}
Our next goal is to construct the target of $\psi$.  This target should be $\Mod_g(\ell)$-invariant, and the 
obvious $\Mod_g(\ell)$-invariant object associated to $\Mod_g$ is the homology group
$\HH_1(\Sigma_g;\Z/\ell)$.  Letting $p$ be a prime dividing
$\ell$, the quotient $\HH_1(\Sigma_g;\Field_p) \cong \Field_p^{2g}$ of $\HH_1(\Sigma_g;\Z/\ell)$
is also invariant
under $\Mod_g(\ell)$.  To explain how we will use this to construct a target for
$\psi$, we need to introduce Tits buildings.  Fix some $n \geq 2$.
The {\em Tits building} associated to $\SL_n(\Field_p)$, denoted $\Tits_n(\Field_p)$,
is the simplicial complex
whose $r$-simplices are flags
\[0 \subsetneq V_0 \subsetneq V_1 \subsetneq \cdots \subsetneq V_r \subsetneq \Field_p^n.\]
The simplicial complex $\Tits_n(\Field_p)$ is $(n-2)$-dimensional, and the
Solomon--Tits theorem (\cite{SolomonTits}; see also \cite[Theorem IV.5.2]{BrownBuildings}) says that
in fact $\Tits_n(\Field_p)$ is homotopy equivalent to a wedge of $(n-2)$-dimensional spheres.  The
{\em Steinberg module} for $\SL_n(\Field_p)$ is defined to be
\[\St_n(\Field_p) \coloneq \RH_{n-2}(\Tits_n(\Field_p);\Q).\]
This is one of the most important representations of $\SL_n(\Field_p)$; for instance, 
it is the unique nontrivial irreducible representation of $\SL_n(\Field_p)$ whose
dimension is a power of $p$.

\paragraph{Apartments.}
The Solomon--Tits theorem also gives generators for $\St_n(\Field_p)$.  Let
$\fX_n$ be the simplicial complex whose $r$-simplices are increasing sequences
\[\emptyset \subsetneq I_0 \subsetneq I_1 \subsetneq \cdots \subsetneq I_r \subsetneq \{1,\ldots,n\}\]
of sets.  The simplicial complex $\fX_n$ is isomorphic to the boundary of the barycentric subdivision
of an $(n-1)$-simplex, and thus is homeomorphic to an oriented $(n-2)$-sphere; let $[\fX_n]$ be its
fundamental class.  Associated to an ordered sequence
$B = (\vec{v}_1,\ldots,\vec{v}_n)$ of $n$ nonzero vectors in $\Field_p^n$, there is a simplicial map
$\fX_n \rightarrow \Tits_n$ taking the vertex $\emptyset \subsetneq I \subsetneq \{1,\ldots,n\}$ to
$\langle \text{$\vec{v}_i$ $|$ $i \in I$} \rangle$.  The image under this map of $[\fX_n] \in \RH_{n-2}(\fX_n;\Q)$
is the {\em apartment class} $\Apartment_B \in \St_n(\Field_p) = \RH_{n-2}(\Tits_n(\Field_p);\Q)$.  The Solomon--Tits theorem
asserts that $\St_n(\Field_p)$ is generated by apartment classes.

\paragraph{Properties of apartments.}
The following result of Lee--Szczarba \cite{LeeSzczarbaCongruence} 
summarizes a number of basic properties of apartment classes.

\begin{theorem}[{\cite[Theorem 3.1]{LeeSzczarbaCongruence}}]
\label{theorem:apartmentrelations}
Let $n \geq 2$ and let $p$ be a prime.  The following then hold for all ordered sequences
$B = (\vec{v}_1,\ldots,\vec{v}_n)$ of nonzero vectors in $\Field_p^n$.
\begin{compactenum}
\item We have $\Apartment_B \neq 0$ if and only if $B$ is a basis for $\Field_p^n$.
\item Let $\sigma$ be a permutation of $\{1,\ldots,n\}$ and let $\sigma(B) = (\vec{v}_{\sigma(1)},\ldots,\vec{v}_{\sigma(n)})$.
Then $\Apartment_B = (-1)^{|\sigma|} \Apartment_{\sigma(B)}$.
\item For nonzero scalars $c_1,\ldots,c_n \in \Field_p^{\ast}$, 
let $B' = (c_1 \vec{v}_1,\ldots,c_n \vec{v}_n)$.  Then $\Apartment_B = \Apartment_{B'}$.
\item For an ordered sequence $C = (\vec{w}_1,\ldots,\vec{w}_{n+1})$ of nonzero vectors in $\Field_p^n$, define
$C_i = (\vec{w}_1,\ldots,\widehat{\vec{w}_i},\ldots,\vec{w}_{n+1})$.  Then
$\sum_{i=1}^{n+1} (-1)^i \Apartment_{C_i} = 0$.
\end{compactenum}
What is more, these generate all relations between apartment classes in $\St_n(\Field_p)$ in the sense that
$\St_n(\Field_p)$ is the $\Q$-vector space with generators the set of formal symbols
\[\Set{$\Apartment_B$}{$B$ an ordered sequence of $n$ nonzero vectors in $\Field_p^n$}\]
and relations the four relations listed above.
\end{theorem}

\begin{remark}
Lee--Szczarba stated their result differently than we have, but the two formulations are equivalent.  The
only non-obvious relation among those above is the fourth one; see Figure \ref{figure:apartmentscancel} for
an explanation of it.
\end{remark}

\Figure{figure:apartmentscancel}{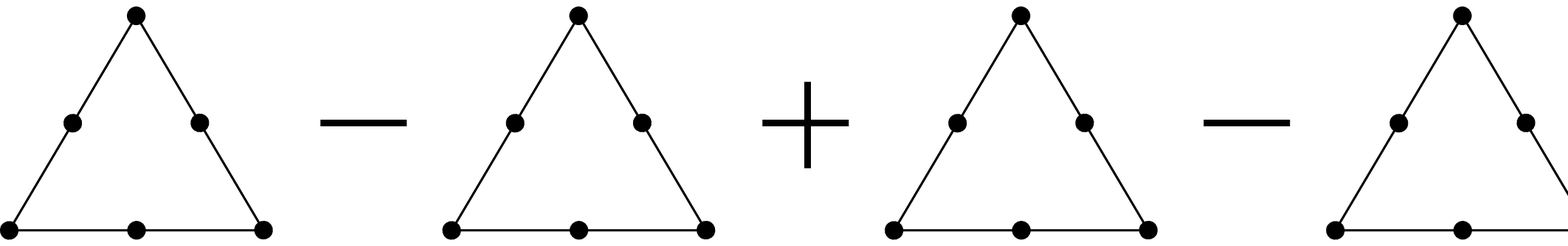}{
As we illustrate here in the case $n=3$, the apartment
classes $\Apartment_{C_i}$ can be placed on the boundary of an $n$-dimensional simplex.  In the picture, the vertices labeled with the vectors $\vec{w}_i$ represent the lines
spanned by the $\vec{w}_i$ while the unlabeled vertices represent the $2$-dimensional subspaces spanned
by the vectors on their two neighbors.}{42}

\paragraph{The map on chains.}
If $V$ is a vector space over $\Q$, then Lemma \ref{lemma:arcdescription} implies that
a homomorphism
\[\psi\colon \St(\Sigma_g) \rightarrow V\]
can be constructed by writing down a homomorphism
\[\phi\colon \CC_{2g-1}(\Arc_g^1,\Arc_g^1(\infty);\Q) \rightarrow V\]
that 
vanishes on the image of 
\[\partial\colon \CC_{2g}(\Arc_g^1;\Q) \rightarrow \CC_{2g-1}(\Arc_g^1,\Arc_g^1(\infty);\Q).\]
Recall that $p$ is a fixed prime dividing $\ell$.  Fix
an identification of $\HH_1(\Sigma_g;\Field_p)$ with
$\Field_p^{2g}$ that takes the algebraic intersection form on $\HH_1(\Sigma_g;\Field_p)$ to  
the standard symplectic form on $\Field_p^{2g}$.
Given an oriented closed curve $\gamma$ on $\Sigma_g$, let $[\gamma]_p$ be the associated
element of $\HH_1(\Sigma_g;\Field_p) = \Field_p^{2g}$.
We can define a map $\tphi\colon \CC_{2g-1}(\Arc_g^1;\Q) \rightarrow \St_{2g}(\Field_p)$ as follows.
\begin{compactitem}
\item Let $\sigma$ be an oriented $(2g-1)$-simplex of $\Arc_g^1$.  Write $\sigma = \{\alpha_0,\ldots,\alpha_{2g-1}\}$, and 
orient each curve
$\alpha_i$ in an arbitrary way.  If $[\alpha_i]_p = 0$ for some $i$, then define $\tphi(\sigma)=0$.  Otherwise, 
set $B = ([\alpha_0]_p,\ldots,[\alpha_{2g-1}]_p)$ and define $\tphi(\sigma) = \Apartment_B$.  
\end{compactitem}
The
second and third relations in Theorem \ref{theorem:apartmentrelations} show that $\tphi$ does not depend on any of the
choices we have made.  Moreover, the first relation in Theorem \ref{theorem:apartmentrelations} combined with
Lemma \ref{lemma:recognizespan} shows that $\tphi(\sigma) = 0$ 
if there is any non-simply connected component in the complement of the $\alpha_i$, so $\tphi$ vanishes on
$\CC_{2g-1}(\Arc_g^1(\infty))$ and thus descends to a map 
\[\phi\colon \CC_{2g-1}(\Arc_g^1,\Arc_g^1(\infty);\Q) \rightarrow \St_{2g}(\Field_p).\]

\paragraph{A problem.}
Unfortunately, $\phi$ does not vanish on the image of $\partial$.
Consider an oriented $2g$-simplex $\eta = \{\beta_0,\ldots,\beta_{2g}\}$ of $\Arc_g^1$.  
Arbitrarily orient each $\beta_i$.  There are two cases where $\phi(\partial(\eta)) = 0$.
\begin{compactitem}
\item If all the $[\beta_i]_p$ are
nonzero, then $\phi$ takes $\partial(\eta)$ to a relation in $\St_{2g}(\Field_p)$ like in the fourth 
relation in Theorem \ref{theorem:apartmentrelations}, and in particular $\phi(\partial(\eta)) = 0$.  
\item If more than one of the $[\beta_i]_p$ are zero, then every simplex in $\partial(\eta)$ contains a curve
whose homology class is $0$, so $\phi$ takes every simplex in $\partial(\eta)$ to $0$.
\end{compactitem}
However, if exactly one of the $[\beta_i]_p$ is zero, then we might have $\phi(\partial(\eta)) \neq 0$.  
To understand precisely what is going on, assume that $[\beta_j]_p = 0$
for some $j$ but that $[\beta_i]_p \neq 0$ for all $i \neq j$.  Let $C_j$ be the result of deleting
$[\beta_j]$ from $([\beta_0]_p,\ldots,[\beta_{2g}])$.  We then have $\phi(\partial(\eta)) = (-1)^j \Apartment_{C_j}$.
It is certainly possible for $\Apartment_{C_j} \neq 0$.  To correct for this, we make the following observation
about the homology classes making up $C_j$.
Since $[\beta_j]_p = 0$, the curve $\beta_j$ must be a separating
simple closed curve (if it were nonseparating, then it would represent a primitive element 
of $\HH_1(\Sigma_g;\Z)$, and hence its homology class would remain nonzero when reduced modulo $p$).  Letting
$S_1$ and $S_2$ be the subsurfaces on either side of $\beta_j$, we obtain a splitting
\[\HH_1(\Sigma_g;\Field_p) = \HH_1(S_1;\Field_p) \oplus \HH_1(S_2;\Field_p)\]
that is orthogonal with respect to the algebraic intersection form.  Our observation then is that for each 
$i \neq j$, the homology class
$[\beta_i]_p$ lies in either $\HH_1(S_1;\Field_p)$ or $\HH_1(S_2;\Field_p)$.

\paragraph{Separated apartments.}
This motivates the following definition.  An apartment $\Apartment_B$ of $\St_{2g}(\Field_p)$ is a
{\em separated apartment} if there exists a nontrivial splitting $\Field_p^{2g} = S_1 \oplus S_2$ that is
orthogonal with respect to the standard symplectic form on $\Field_p^{2g}$ such that each $\vec{v}$ occurring
in $B$ lies in either $S_1$ or $S_2$.  The following lemma shows that this implies that
the symplectic form on $\Field_p$ restricts to a symplectic form on each $S_i$.

\begin{lemma}
\label{lemma:orthogonalsymplectic}
Let $V$ be a symplectic vector space over a field and let $W_1,\ldots,W_k \subset V$ be subspaces
such that $V$ is the internal direct sum of the $W_i$.  Assume that the $W_i$ are pairwise orthogonal
to each other.  Then each $W_i$ is a symplectic subspace of $V$, i.e.\ the restriction of the
symplectic form on $V$ to $W_i$ is a symplectic form.
\end{lemma}
\begin{proof}
Fix some $1 \leq i \leq k$ and let $r \in W_i$ be orthogonal to every vector in $W_i$.  We must prove that $r=0$.
Since $W_i$
is orthogonal to every $W_{i'}$ with $i' \neq i$, the vector $r$ is also orthogonal to every such $W_{i'}$.
We deduce that $r$ is orthogonal to every vector in $V$, and thus that $r=0$.
\end{proof}

\noindent
The {\em separated subspace} of $\St_{2g}(\Field_p)$, denoted $\StSep_{2g}(\Field_p)$, is
the subspace spanned by separated
apartments.  While $\StSep_{2g}(\Field_p)$ is not preserved by the action of $\SL_{2g}(\Field_p)$, it is
preserved by the subgroup $\Sp_{2g}(\Field_p)$ of $\SL_{2g}(\Field_p)$.  
By the previous paragraph, for each $2g$-simplex $\eta$ of $\Arc_g^1$, either
$\phi(\partial(\eta)) = 0$ or $\phi(\partial(\eta)) \in \StSep_{2g}(\Field_p)$.  Define
\[\StNS_{2g}(\Field_p) = \St_{2g}(\Field_p) / \StSep_{2g}(\Field_p),\] 
so $\StNS_{2g}(\Field_p)$ is
an $\Sp_{2g}(\Field_p)$-representation.  We deduce that $\phi$ descends
to a well-defined map 
\[\psi\colon \St(\Sigma_g) \cong \Coker(\partial) \longrightarrow \StNS_{2g}(\Field_p).\]
By its construction, $\psi$ is equivariant under the action of the whole mapping class group
$\Mod_g$ and is invariant under the action of the subgroup $\Mod_g(\ell)$.  

\paragraph{Putting everything together.}
The main result of this section is the following, which is part of the assertion of
Theorem \ref{maintheoremprime:cohomology}.

\begin{proposition}
\label{proposition:surjective}
Fix $g, \ell \geq 2$.  Let $p$ be a prime dividing $\ell$.  Then the map
\[\psi\colon \St(\Sigma_g) \longrightarrow \StNS_{2g}(\Field_{p})\]
constructed above is surjective.
\end{proposition}
\begin{proof}
For $x \in \St_{2g}(\Field_p)$, write $\NS{x}$ for the image of $x$ in $\StNS_{2g}(\Field_p)$.
Also, let $\omega(\cdot,\cdot)$ be the standard symplectic form on $\Field_p^{2g}$.
Our proof will have two steps.

\begin{stepa}
\label{stepa:broaddus}
The vector space $\StNS_{2g}(\Field_{p})$ is spanned by the set of
$\NS{\Apartment_B}$ such that $B = (\vec{v}_1,\ldots,\vec{v}_{2g})$ satisfies the
following two conditions:
\begin{compactitem}
\item $\omega(\vec{v}_i,\vec{v}_{i+1}) = 1$ for $1 \leq i < 2g$, and
\item $\omega(\vec{v}_i,\vec{v}_j) = 0$ for $1 \leq i,j \leq 2g$ such that $|i-j|>1$.
\end{compactitem}
\end{stepa}
\begin{proof}[Proof of Step \ref{stepa:broaddus}]
The proof of this step is inspired by the proof of \cite[Proposition 4.6]{BroaddusSteinberg}.
For $1 \leq k \leq 2g$, let $V_k$ be the subspace of $\StNS_{2g}(\Field_p)$ spanned
by the set of $\NS{\Apartment_B}$ such that $B = (\vec{v}_1,\ldots,\vec{v}_{2g})$ satisfies
the following three conditions:
\begin{compactitem}
\item $\omega(\vec{v}_i,\vec{v}_{i+1}) = 1$ for $1 \leq i < k$, and
\item $\omega(\vec{v}_i,\vec{v}_j) = 0$ for $1 \leq i,j \leq k$ such that $|i-j|>1$, and
\item $\omega(\vec{v}_i,\vec{v}_j) = 0$ for $1 \leq i < k$ and $k<j\leq 2g$.
\end{compactitem}
We will prove that 
$V_k$ spans $\StNS_{2g}(\Field_p)$ by induction on $k$; the case $k=2g$ is precisely the
claim we are trying to prove.  The base case $k=1$ is trivial,
so assume that $1 < k \leq 2g$ and that $V_{k-1}$ spans $\StNS_{2g}(\Field_p)$.
Consider an apartment $\Apartment_B$ such that $B$ satisfies the three conditions
above for $k-1$, and thus $\NS{\Apartment_B} \in V_{k-1}$.  It is enough to prove
that $\NS{\Apartment_B} \in V_k$.
Since the symplectic form $\omega(\cdot,\cdot)$ is nondegenerate, we can find
$\vec{w} \in \Field_p^{2g}$ such that the following two conditions hold:
\begin{compactitem}
\item $\omega(\vec{w},\vec{v}_1) = 1$, and
\item $\omega(\vec{w},\vec{v}_i) = 0$ for $2 \leq i \leq 2g$.
\end{compactitem}
Define $C = (\vec{w},\vec{v}_1,\vec{v}_2,\ldots,\vec{v}_{2g})$, and for
$1 \leq i \leq 2g+1$, let $C_i$ be the result of deleting the $i^{\text{th}}$ vector
in $C$.  The fourth relation in Theorem \ref{theorem:apartmentrelations} says that
\begin{equation}
\label{eqn:makealine}
\sum_{i=1}^{2g+1} (-1)^{i} \Apartment_{C_i} = 0
\end{equation}
in $\St_{2g}(\Field_p)$.  We now make the following observations.
\begin{compactitem}
\item $C_1 = B$.
\item For $2 \leq i \leq k$, we claim that $\Apartment_{C_i} \in \StSep_{2g}(\Field_p)$ and hence
that $\NS{\Apartment_{C_i}} = 0$.  This claim is trivial if $\Apartment_{C_i} = 0$, so assume it is nonzero.
We have 
\[C_i = (\vec{w},\vec{v}_1,\ldots,\widehat{{\vec{v}_{i-1}}},\ldots,\vec{v}_{2g}).\]  
Since we are assuming that $\Apartment_{C_i} \neq 0$, these vectors are a basis for $\Field_p^{2g}$.  Let
$S_1 \subset \Field_p^{2g}$ be the span of the first $(i-1)$ vectors in $C_i$ and let
$S_2 \subset \Field_p^{2g}$ be the span of the remaining ones.  We thus have $\Field_p^{2g} = S_1 \oplus S_2$.
The key observation now is that by construction $S_1$ is orthogonal to $S_2$.  By Lemma \ref{lemma:orthogonalsymplectic}, this implies that
each $S_i$ is a symplectic subspace of $\Field_p^{2g}$.  From this, we see that
$\Apartment_{C_i} \in \StSep_{2g}(\Field_p)$, as claimed. 
\item For $k<i\leq 2g+1$, we have $\NS{\Apartment_{C_i}} \in V_k$
\end{compactitem}
Combining these three observations with \eqref{eqn:makealine}, we see that
$\NS{\Apartment_B} \in V_k$, as desired.
\end{proof}

To prove Proposition \ref{proposition:surjective}, it is enough to show that
the image of $\psi$ contains all the generators identified in Step \ref{stepa:broaddus}, which
is the content of our next step.

\begin{stepa}
\label{stepa:realize}
Consider an apartment $\Apartment_B$ such that $B = (\vec{v}_1,\ldots,\vec{v}_{2g})$ satisfies
the two conditions stated in Step \ref{stepa:broaddus}.  Then there exists an element
$x \in \St(\Sigma_g)$ such that $\psi(x) = \NS{\Apartment_B}$.
\end{stepa}
\begin{proof}[Proof of Step \ref{stepa:realize}]
The group $\Sp_{2g}(\Field_p)$ acts transitively on ordered bases $B$ for
$\Field_p^{2g}$ that satisfy the two conditions in Step \ref{stepa:broaddus}.  Since the map
$\Mod_g \rightarrow \Sp_{2g}(\Field_p)$ coming from the action on $\HH_1(\Sigma_g;\Field_p)$
is surjective and the homomorphism $\psi$ is equivariant with respect to the mapping class group
actions on $\St(\Sigma_g)$ and $\StNS_{2g}(\Field_p)$ (the latter action coming from the
surjection $\Mod_g \rightarrow \Sp_{2g}(\Field_p)$),
we see that it is enough to find an element
$x \in \St(\Sigma_g)$ such that $\psi(x) = \NS{\Apartment_C}$, where
$C = (\vec{w}_1,\ldots,\vec{w}_{2g})$ is {\em some} ordered basis for $\Field_p^{2g}$
satisfying the two conditions in Step \ref{stepa:broaddus} (with $\vec{w}_i$ swapped for
$\vec{v}_i$).  
Examining the construction of $\psi$, we see that we can take $x$ to be
the image in $\St(\Sigma_g)$ of the $(2g-1)$-simplex
\[\sigma = \{\alpha_1,\gamma_1,\alpha_2,\gamma_2,\ldots,\alpha_g,\gamma_g\}\]
of $\Arc_g^1$ depicted in Figure \ref{figure:arcsimplex}.
\end{proof}

This completes the proof of Proposition \ref{proposition:surjective}.
\end{proof}

\Figure{figure:arcsimplex}{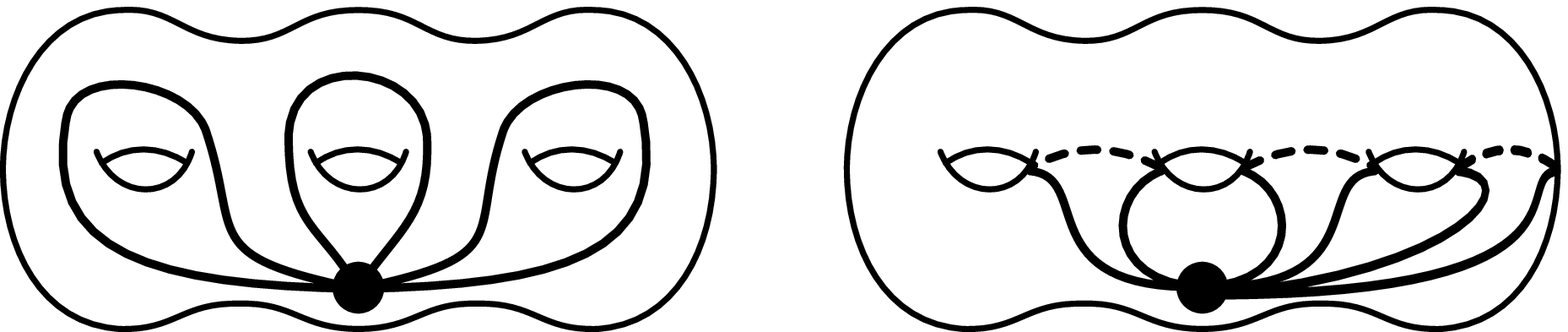}{A picture for $g=3$ of the simplex
$\sigma = \{\alpha_1,\gamma_1,\alpha_2,\gamma_2,\ldots,\alpha_g,\gamma_g\}$ of $\Arc_g^1$ used in the proof of Proposition \ref{proposition:surjective}.  The
pattern for higher $g$ is evident.}{80}

To derive Theorem \ref{maintheoremprime:cohomology} from Proposition \ref{proposition:surjective}, it is
enough to prove the following proposition.

\begin{proposition}
\label{proposition:nosepbig}
Let $g \geq 1$ and let $p$ be a prime.  Then
\[\dim_{\Q} \StNS_{2g}(\Field_p) = \frac{|\Sp_{2g}(\Field_p)|}{g(p^{2g}-1)}.\]
\end{proposition}

We will prove Proposition \ref{proposition:nosepbig} in \S \ref{section:bounds}.  This is preceded by a careful study of
the structure of $\St_{2g}(\Field_p)$ as an $\Sp_{2g}(\Field_p)$ representation in \S \ref{section:decompose}.

\section{Decomposing the nonseparated building}
\label{section:decompose}

In preparation for proving Proposition \ref{proposition:nosepbig} in \S \ref{section:bounds}, 
this section is devoted to understanding the $\Sp_{2g}(\Field_p)$-module structure of 
$\St_{2g}(\Field_p)$.  We begin with \S \ref{section:decomposelinearalgebra}, which is a
preliminary section containing a result about representations of posets.  Our main result
is stated and proved in \S \ref{section:decomposemain}.  This proof depends on three
lemmas whose proofs are postponed; these proofs are in \S \ref{section:decomposeproduct}, \S \ref{section:decomposeprojection},
and \S \ref{section:decomposecrossproj}.

\subsection{Representations of posets}
\label{section:decomposelinearalgebra}

We begin with two definitions.

\begin{definition}
A {\em linear representation} $\bV$ of a poset $(\cP,\preceq)$ is a vector space $\bV$
equipped with subspaces $\bV(A)$ for all $A \in \cP$ such that $\bV(A) \subseteq \bV(B)$
for all $A,B \in \cP$ satisfying $A \preceq B$.  The vector space $\bV$ is the {\em underlying
vector space} of the representation, and the {\em span} of the representation is the span
of the subspaces $\Set{$\bV(A)$}{$A \in \cP$}$.
\end{definition}

\begin{definition}
Let $\bV$ be a linear representation of a poset $(\cP,\preceq)$ and let $A \in \cP$.
The {\em decomposable subspace} of $\bV(A)$, denoted $\bVdec(A)$, is the span
of the set of subspaces $\Set{$\bV(A')$}{$A' \preceq A$, $A' \neq A$}$ of $\bV(A)$.  The
{\em indecomposable part} of $\bV(A)$, denoted $\bVindec(A)$, is the quotient
$\bV(A) / \bVdec(A)$.
\end{definition}

Let $\bV$ be a linear representation of a poset $(\cP,\preceq)$ and let $W$ be
the span of $\bV$.  The main result of this section gives a condition that ensures
that
\[W \cong \prod_{A \in \cP} \bVindec(A).\]
To construct this isomorphism, we will use the following.

\begin{definition}
Let $\bV$ be a linear representation of a poset $(\cP,\preceq)$.  A {\em projection
system} $\pi$ for $\bV$ consists of linear maps $\pi_A\colon \bV \rightarrow \bV(A)$
for each $A \in \cP$ such that $\pi_A(a) = a$ for all $a \in \bV(A)$.  Given
a projection system $\pi$, for $A \in \cP$
define $\piindec_A\colon \bV \rightarrow \bVindec(A)$ to be the composition
\[\bV \stackrel{\pi_A}{\longrightarrow} \bV(A) \longrightarrow \bV(A) / \bVdec(A) = \bVindec(A).\]
\end{definition}

Our main result is then as follows.

\begin{proposition}
\label{proposition:linearalgebra}
Let $\bV$ be a linear representation of a finite poset $(\cP,\preceq)$, let $W$ be the
span of $\bV$, and let $\pi$ be a projection system for $\bV$.  Assume that
\begin{equation}
\label{eqn:projasm}
\pi_A(\bV(B)) \subset \bVdec(A) \quad \text{for all $A, B \in \cP$ satisfying
$A \npreceq B$}.
\end{equation}
Then the map
\[W \longrightarrow \prod_{A \in \cP} \bVindec(A)\]
obtained by taking the direct product of the restrictions of the
$\piindec_A\colon \bV \rightarrow \bVindec(A)$ to $W$ is an isomorphism.
\end{proposition}
\begin{proof}
The proof will be by induction on the cardinality of the finite poset $\cP$.  The base
case where $\cP$ consists of one element is trivial, so assume that $\cP$ has at least
two elements and that the proposition is true for all smaller posets.  Let $M \in \cP$
be a maximal element, that is, an element such that the only $A \in \cP$
with $M \preceq A$ is $A = M$.  Set $\cP' = \cP \setminus \{M\}$.  Since $M$ is maximal,
the assumption \eqref{eqn:projasm} implies that $\pi_M(\bV(A')) \subset \bVdec(M)$ for all $A' \in \cP'$, and
hence $\piindec_M(\bV(A')) = 0$ for all $A' \in \cP'$.  Letting $W'$ be the span
of the restriction of $\bV$ to $\cP'$, we deduce that $W' \subset \ker(\piindec_M)$.
Since $W = W' + \bV(M)$ and
\[\ker((\piindec_M)|_{\bV(M)}) = \ker(\bV(M) \rightarrow \bVindec(M)) = \bVdec(M) \subset W',\]
it follows that $W'$ equals the kernel of the restriction of $\piindec_M$ to $W$.  In
other words, we have a short exact sequence
\[\begin{CD}
0 @>>> W' @>>> W @>{\piindec_M}>> \bVindec(M) @>>> 0.
\end{CD}\]
This short exact sequence fits into a commutative diagram of short exact sequences
\begin{equation}
\label{eqn:decomposev}
\begin{CD}
0 @>>> W'   @>>> W    @>{\piindec_M}>> \bVindec(M) @>>> 0\\
@.     @VVV      @VVV             @VV{=}V          @.\\
0 @>>> {\displaystyle \prod_{A' \in \cP'} \bVindec(A')} @>>> {\displaystyle \prod_{A \in \cP} \bVindec(A)} @>>> \bVindec(M) @>>> 0
\end{CD}
\end{equation}
whose left and central vertical arrows are the direct products of the restrictions of the
relevant $\piindec_A$.
Since $M$ is maximal, for $A' \in \cP'$ the decomposable subspace and indecomposable part
of $\bV(A')$ is
the same whether $\bV$ is considered a linear representation of $\cP$ or of its
subposet $\cP'$.  We can thus apply our inductive hypothesis to see that the
left hand vertical arrow in \eqref{eqn:decomposev} is an isomorphism.  The five lemma
therefore implies that the central vertical arrow in \eqref{eqn:decomposev} is an
isomorphism, as desired.
\end{proof}

\subsection{The decomposition}
\label{section:decomposemain}

Let $g \geq 1$ and let $p$ be a prime.
All vector spaces in this section are finite-dimensional 
vector spaces over $\Field_p$.  Before stating our main result, we need some preliminaries.

\paragraph{Abstract vector spaces.}
Given a vector space $V$, let $\Tits(V)$ be the Tits building and 
$\St(V)$ be the Steinberg module associated to $\SL(V)$.  Thus if $n = \dim_{\Field_p}(V)$, then $\St(V) = \RH_{n-2}(\Tits(V);\Q)$.  
If $V$ is a symplectic vector space, then let $\StSep(V)$ denote
the separated subspace of $\St(V)$ as defined in \S \ref{section:maptobuilding}.  Also, define
$\StNS(V) = \St(V) / \StSep(V)$.  Finally, if $V$ is a vector space and $V_1,V_2$
are linearly independent subspaces of $V$, then we will write $V_1 \boxplus V_2$
for the internal direct sum of $V_1$ and $V_2$.

\paragraph{Unordered tensor products.}
Given vector spaces $V_1,\ldots,V_k$, the {\em unordered tensor product} of the $V_i$, denoted
$\Sym(V_1,\ldots,V_k)$, is defined as follows.  Let $\cSym(k)$ be the symmetric group on $k$ letters.  Set
\[\cO = \Set{$(W_1,\ldots,W_k)$}{there exists $\sigma \in \cSym(k)$ such that $W_i = V_{\sigma(i)}$ for $1 \leq i \leq k$}\]
and
\[\hSym(V_1,\ldots,V_k) = \bigoplus_{(W_1,\ldots,W_k) \in \cO} W_1 \otimes \cdots \otimes W_k.\]
The group $\cSym(k)$ acts on $\cO$ in the obvious way, and this induces an action of $\cSym(k)$ on
$\hSym(V_1,\ldots,V_k)$.  By definition, $\Sym(V_1,\ldots,V_k)$ is the coinvariants of the $\cSym(k)$-action on
$\hSym(V_1,\ldots,V_k)$.
As its name suggests, $\Sym(V_1,\ldots,V_k)$ is exactly like the 
ordinary tensor product, but without a distinguished ordering of the factors; in particular, we have
an equality
\[\Sym(V_1,\ldots,V_k) = \Sym(V_{\sigma(1)},\ldots,V_{\sigma(k)})\]
for all $\sigma \in \cSym(k)$.  The composition
\[V_1 \otimes \cdots \otimes V_k \hookrightarrow \hSym(V_1,\ldots,V_k) \rightarrow \Sym(V_1,\ldots,V_k)\]
is an isomorphism.  Given elements $\vec{v}_i \in V_i$
for $1 \leq i \leq k$, we will write $\Sym(\vec{v}_1,\ldots,\vec{v}_k)$ for the image
in $\Sym(V_1,\ldots,V_k)$ of $\vec{v}_1 \otimes \cdots \otimes \vec{v}_k$.

\paragraph{Symplectic splittings.}
Endow $\Field_p^{2g}$ with the standard symplectic form.
A {\em symplectic splitting} of $\Field_p^{2g}$ is an unordered set $S=\{S_1,\ldots,S_k\}$
of nonzero subspaces of $\Field_p^{2g}$ that are pairwise orthogonal to each other under the symplectic form such that
$\Field_p^{2g} = S_1 \boxplus \cdots \boxplus S_k$.
Using Lemma \ref{lemma:orthogonalsymplectic}, this implies that the symplectic form on 
$\Field_p^{2g}$ restricts to a symplectic form on each $S_i$.  Given
a symplectic splitting $S=\{S_1,\ldots,S_k\}$ of $\Field_p^{2g}$, we define
\[\St(S) = \Sym(\St(S_1),\ldots,\St(S_k))\]
and
\[\StNS(S) = \Sym(\StNS(S_1),\ldots,\StNS(S_k)).\]
There is a projection
$\St(S) \rightarrow \StNS(S)$ whose kernel is the span of the subspaces
\[\Set{$\Sym(\StSep(S_i), \St(S_1),\ldots,\widehat{\St(S_i)},\ldots,\St(S_k))$}{$1 \leq i \leq k$}\]
of $\St(S)$.  Denote this kernel by $\StSep(S)$, so we have a short exact sequence
\[0 \longrightarrow \StSep(S) \longrightarrow \St(S) \longrightarrow \StNS(S) \longrightarrow 0.\]  
Let $\Sp(S)$ be the stabilizer of $S$ in $\Sp_{2g}(\Field_p)$.  The vector spaces $\St(S)$ and $\StSep(S)$ and
$\StNS(S)$ are representations of $\Sp(S)$, and the above short exact sequence is an exact
sequence of $\Sp(S)$-representations.

\paragraph{Main theorem, statement.}
We can now state our main theorem.  
Let $\fS_g$ be the set of all symplectic splittings of $\Field_p^{2g}$ (including the $1$-element splitting $\{\Field_p^{2g}\}$).

\begin{theorem}
\label{theorem:decomposition}
Let $g \geq 1$ and let $p$ be a prime.  Then we have a vector space isomorphism
\[\St_{2g}(\Field_p) \cong \bigoplus_{S \in \fS_g} \StNS(S).\]
\end{theorem}

\begin{remark}
\label{remark:branching}
Though we will not need this, it will be clear from our proof that 
the isomorphism in Theorem \ref{theorem:decomposition} is actually an isomorphism
of $\Sp_{2g}(\Field_p)$-representations.  From this point of view, 
Theorem \ref{theorem:decomposition} can be stated in more representation-theoretic terms: letting
$U \subset \fS_g$ be a set containing a single representative of each $\Sp_{2g}(\Field_p)$-orbit, we have an isomorphism
\[\St_{2g}(\Field_p) \cong \bigoplus_{S \in \fS_g} \StNS(S) \cong \bigoplus_{S \in U} \Ind_{\Sp(S)}^{\Sp_{2g}(\Field_p)} \StNS(S).\]
of $\Sp_{2g}(\Field_p)$-representations.  We do not know if these induced representations are irreducible or not.
\end{remark}

\paragraph{Poset of symplectic splittings.}
We will prove Theorem \ref{theorem:decomposition} using Proposition \ref{proposition:linearalgebra}. 
We start by endowing $\fS_g$ with a poset structure.
Say that $S' \in \fS_g$ is a {\em refinement} of $S \in \fS_g$ if each subspace occurring in
$S$ is the direct sum of a subset of the subspaces occurring in $S'$.  This gives
a partial order on $\fS_g$ where $S' \preceq S$ when $S'$ is a refinement of $S$.

\paragraph{The representation.}
Define $\bV_g = \St_{2g}(\Field_p)$.  We now endow $\bV_g$ with the structure of a representation
of the poset $(\fS_g,\preceq)$.
If $B_1,\ldots,B_k$ are each ordered sequences
of nonzero vectors in $\Field_p^{2g}$, then let $B_1 \cdots B_k$ denote the result
of concatenating the $B_i$.  For an element $S = \{S_1,\ldots,S_k\}$ of $\fS_g$, we define
$\bV_g(S)$ to be the span in $\Sp_{2g}(\Field_p)$ of the set
\[\Set{$\Apartment_{B_1 \cdots B_k}$}{$B_i$ is an ordered sequence of $\dim(S_i)$ nonzero vectors
in $S_i$ for $1 \leq i \leq k$}.\]
Since permuting the vectors forming an apartment class only changes the apartment class by a sign, this
does not depend on the ordering of the $S_i$.  
It is clear that if $S'$ is a refinement of $S$, then $\bV_g(S') \subseteq \bV_g(S)$, so this
defines a linear representation of $\bV_g$.  Moreover, since $\bV_g(\{\Field_p^{2g}\}) = \St_{2g}(\Field_p)$,
the span of this linear representation is $\St_{2g}(\Field_p)$.

\paragraph{Product and projection maps.}
Our next goal is to prove that $\bV_g(S) \cong \St(S)$ for all $S \in \fS_g$.  This will
depend on the following two lemmas whose proofs are postponed.

\begin{lemma}
\label{lemma:productmap}
Let $S = \{S_1,\ldots,S_k\}$ be a symplectic splitting of $\Field_p^{2g}$.  There exists a 
linear map $\inc_S\colon \St(S) \rightarrow \St_{2g}(\Field_p)$ such that if $B_i$ is
an ordered sequence of $\dim(S_i)$ nonzero vectors in $S_i$ for $1 \leq i \leq k$, then
$\inc_S(\Sym(\Apartment_{B_1},\ldots,\Apartment_{B_k})) = \Apartment_{B_1 \cdots B_k}$.
\end{lemma}

\noindent
The maps $\inc_S$ will be called {\em product maps}.

\begin{lemma}
\label{lemma:projectionmapunordered}
Let $S = \{S_1,\ldots,S_k\}$ be a symplectic splitting of $\Field_p^{2g}$.  There
exists a linear map $\pi_S\colon \St_{2g}(\Field_p) \rightarrow \St(S)$ such that if
$B_i$ is an ordered sequence of $\dim(S_i)$ nonzero vectors in $S_i$ for $1 \leq i \leq k$, then
$\pi_S(\Apartment_{B_1 \cdots B_k}) = \Sym(\Apartment_{B_1},\ldots,\Apartment_{B_k})$.
\end{lemma}

\noindent
The maps $\pi_S$ will be called {\em projection maps}.  
The proofs of Lemmas \ref{lemma:productmap} and \ref{lemma:projectionmapunordered} can be found in 
\S \ref{section:decomposeproduct} and \S \ref{section:decomposeprojection}, respectively.

\paragraph{Identifying the representation.}
We now prove the following.

\begin{lemma}
\label{lemma:identifyrep}
Let the notation be as above.  The following then hold for all $S \in \fS_g$.
\begin{compactenum}
\item The product map $\inc_S\colon \St(S) \rightarrow \bV_g$ is an injection with image $\bV_g(S)$.
\item Identifying $\bV_g(S)$ with $\St(S)$ via $\inc_S$, the following hold.
\begin{compactenum}
\item The projection map $\pi_S\colon \bV_g \rightarrow \St(S)$
restricts to the identity map from $\bV_g(S) = \St(S)$ to $\St(S)$.
\item We have $\bVdec_g(S) = \StSep(S) \subset \St(S)$.
\end{compactenum}
\end{compactenum}
\end{lemma}
\begin{proof}
Write $S = \{S_1,\ldots,S_k\}$.
The formula in the statement of Lemma \ref{lemma:productmap} says that $\inc_S$ takes generators of $\St(S)$ to
generators of $\bV_g(S)$, so the image of $\inc_S$ is $\St(S)$.  If $B_i$
is an ordered sequence of $\dim(S_i)$ nonzero vectors in $S_i$ for $1 \leq i \leq k$, then 
\[\pi_S(\inc_S(\Sym(\Apartment_{B_1},\ldots,\Apartment_{B_k}))) = \pi_S(\Apartment_{B_1 \cdots B_k}) = \Sym(\Apartment_{B_1},\ldots,\Apartment_{B_k}).\]
It follows that $\pi_S \circ \inc_S$ is the identity.  Conclusions 1 and 2a follow.
To see that $\bVdec_g(S) = \StSep(S)$ (Conclusion 2b), observe that by definition
$\StSep(S)$ is generated by the set $\Lambda$ of all $\Sym(\Apartment_{B_1},\ldots,\Apartment_{B_k})$, where
each $B_i$ is an ordered sequence of $\dim(S_i)$ nonzero vectors in $S_i$ and where at least one of the $\Apartment_{B_i}$
lies in $\StSep(S_i)$.  As $S'$ ranges over all proper refinements of $S$, these $\Lambda$ are also the 
images under $\inc_{S'}$ of the generators of $\St(S')$.  The equality $\bVdec_g(S) = \StSep(S)$ follows.
\end{proof}

\paragraph{The key technical lemma.}
The following lemma, whose proof is postponed, is perhaps the key technical lemma that goes into proving
Theorem \ref{theorem:decomposition}.

\begin{lemma}
\label{lemma:decomposecrossproj}
Let $S,S' \in \fS_g$ be such that $S \npreceq S'$.  Then $\pi_S(\bV_g(S')) \subset \StSep(S)$.
\end{lemma}

\noindent
The proof of Lemma \ref{lemma:decomposecrossproj} can be found in \S \ref{section:decomposecrossproj}.

\paragraph{Putting it all together.}
All the pieces are now in place to prove Theorem \ref{theorem:decomposition} (modulo the postponed Lemmas
\ref{lemma:productmap}, \ref{lemma:projectionmapunordered}, and \ref{lemma:decomposecrossproj}, which are proved
in the next three subsections).

\begin{proof}[Proof of Theorem \ref{theorem:decomposition}]
Let $\bV_g$ be the linear representation of $(\fS_g,\preceq)$ discussed above.
Using Conclusion 1 of Lemma \ref{lemma:identifyrep}, we will identify $\bV_g(S)$ with $\St(S)$ via the product
map $\inc_S$ for all $S \in \fS_g$.  Conclusion 2a of that lemma says that we can define a projection system
$\pi$ for $\bV_g$ by letting $\pi_S$ be the projection map for all $S \in \fS_g$.  Lemma \ref{lemma:decomposecrossproj}
says that $\bV_g$ together with the projection system $\pi$ satisfies the hypotheses of Proposition 
\ref{proposition:linearalgebra}.  Since the span of $\bV_g$ is all of $\bV_g = \St_{2g}(\Field_p)$, that
proposition implies that the direct product
\[\St_{2g}(\Field_p) \longrightarrow \prod_{S \in \fS_g} \bVindec_g(S)\]
of the maps
\[\piindec_S\colon \St_{2g}(\Field_p) = \bV_g \rightarrow \bVindec_g(S)\]
is an isomorphism.  For $S \in \fS_g$, Conclusion 2b of Lemma \ref{lemma:identifyrep} says
that $\bVdec_g(S) = \StSep(S) \subset \St(S)$, so $\bVindec_g(S) = \St(S)/\StSep(S) = \StNS(S)$.  The theorem follows.
\end{proof}

\subsection{Product maps: the proof of Lemma \ref{lemma:productmap}}
\label{section:decomposeproduct}

This section is devoted to the proof of Lemma \ref{lemma:productmap}.

\begin{proof}[Proof of Lemma \ref{lemma:productmap}]
We first recall the statement.  Let $p$ be a prime, let $g \geq 1$, and 
let $S = \{S_1,\ldots,S_k\}$ be a symplectic splitting of $\Field_p^{2g}$.  We must prove that there exists a
linear map $\inc_S\colon \St(S) \rightarrow \St_{2g}(\Field_p)$ such that if $B_i$ is
an ordered sequence of $\dim(S_i)$ nonzero vectors in $S_i$ for $1 \leq i \leq k$, then
$\inc_S(\Sym(\Apartment_{B_1},\ldots,\Apartment_{B_k})) = \Apartment_{B_1 \cdots B_k}$.
Set $n_i = \dim_{\Field_p}(S_i)$, so $n_1+\cdots+n_k = 2g$.  Since each $S_i$ is a symplectic subspace
of $\Field_p^{2g}$, each $n_i$ is even.  Since permuting the vectors forming an apartment class only changes
the apartment class by a sign, this implies that the
indicated formula for $\inc_S$ does not depend on the ordering of the $S_i$, and is thus well-defined.

Theorem \ref{theorem:apartmentrelations} gives presentations for the $\St(S_i)$ and hence for $\St(S)$.  We must
check that the indicated formula take the relations for $\St(S)$ to relations for $\St_{2g}(\Field_p)$.  The only
nonobvious relation is the one coming from the fourth relation in Theorem \ref{theorem:apartmentrelations}, which goes
as follows.  To simplify our notation, we will check this relation in the $\St(S_1)$-factor of $\St(S)$; the other
verifications are similar.  Let $C$ be an ordered sequence of $n_1+1$ nonzero vectors in $S_1$, and for $2 \leq i \leq k$
let $B_i$ be an ordered sequence of $n_i$ nonzero vectors in $S_i$.  For $1 \leq i \leq n_1+1$, let $C_i$ be the result
of deleting the $i^{\text{th}}$ vector from $C$.  We then get a relation
\[\sum_{i=1}^{n_1+1} (-1)^i \Sym(\Apartment_{C_i}, \Apartment_{B_2},\ldots,\Apartment_{B_k}) = 0\]
in $\St(S)$.  We must check that $\inc_S$ takes this to a relation in $\St_{2g}(\Field_p)$, i.e.\ we must check that
\begin{equation}
\label{eqn:producttocheck}
\sum_{i=1}^{n_1+1} (-1)^i \Apartment_{C_i B_2 \cdots B_k} = 0.
\end{equation}
Let $D = C B_2 \cdots B_k$, so $D$ is a sequence of $2g+1$ nonzero vectors
in $\Field_p^{2g}$.  For $1 \leq i \leq 2g+1$, let $D_i$ be the result of deleting the $i^{\text{th}}$ vector from $D$.  We thus
have $D_i = C_i B_2 \cdots B_k$ for $1 \leq i \leq n_1+1$.  The fourth relation in Theorem \ref{theorem:apartmentrelations}
says that
\begin{equation}
\label{eqn:productalmostthere}
0 = \sum_{i=1}^{2g+1} (-1)^i \Apartment_{D_i} = \left(\sum_{i=1}^{n_1+1} (-1)^i \Apartment_{C_i B_2 \cdots B_k}\right) + \left(\sum_{i=n_1+2}^{2g+1} (-1)^i \Apartment_{D_i}\right).
\end{equation}
For $n_1+2 \leq i \leq 2g+1$ the vectors in $D_i$ do not span $\Field_p^{2g}$, so the first relation in
Theorem \ref{theorem:apartmentrelations} implies that $\Apartment_{D_i} = 0$.  Plugging this into
\eqref{eqn:productalmostthere}, we obtain \eqref{eqn:producttocheck}, as desired.
\end{proof}

\subsection{Projection maps: the proof of Lemma \ref{lemma:projectionmapunordered}}
\label{section:decomposeprojection}

Let $p$ be a prime and let $g \geq 1$.
This section is devoted to the proof of Lemma \ref{lemma:projectionmapunordered}.  This requires
some preliminary work.

An {\em ordered symplectic splitting} of $\Field_p^{2g}$
is an ordered sequence $\hS = (S_1,\ldots,S_k)$ of distinct subspaces of $\Field_p^{2g}$ such that
$S := \{S_1,\ldots,S_k\}$ is a symplectic splitting of $\Field_p^{2g}$.  We will say that $\hS$ is an
{\em ordering} of $S$.  Next, if $\hS = (S_1,\ldots,S_k)$ is an ordered symplectic splitting of $\Field_p^{2g}$, then
an ordered sequence $B$ of $2g$ vectors in $\Field_p^{2g}$ is {\em weakly compatible} with $\hS$ if we can write
$B = B_1 \cdots B_k$ such that $B_1 \cdots B_i$ is a basis of $S_1 \boxplus \cdots \boxplus S_i$
for all $1 \leq i \leq k$; here recall that $\boxplus$ denotes the internal direct sum.  
In that case, we will say that $B_1 \cdots B_k$ is the {\em factorization} of $B$ associated
to $S$.  Also, the {\em projection} of $B$ to $\hS$ is the sequence $(\oB_1,\ldots,\oB_k)$, where for
$1 \leq i \leq k$ the sequence $\oB_i$ is obtained by applying the orthogonal projection
$\Field_p^{2g} \rightarrow S_i$ 
to each vector in $B_i$.

\begin{lemma}
\label{lemma:projectionmapordered}
Let $\hS = (S_1,\ldots,S_k)$ be an ordered symplectic splitting of $\Field_p^{2g}$.
Then there is a linear map $\proj_{\hS}\colon \St_{2g}(\Field_p) \rightarrow \St(S)$
such that the following properties hold.  Let $B$ be an ordered sequence of $2g$ nonzero
vectors in $\Field_p^{2g}$.
\begin{compactitem}
\item If $B$ is weakly compatible with $\hS$ and $(\oB_1,\ldots,\oB_k)$ is its projection to $\hS$, then
\[\proj_{\hS}(\Apartment_B) = \Sym(\Apartment_{\oB_1},\ldots,\Apartment_{\oB_k}).\]
\item If no reordering of the vectors in $B$ is weakly compatible with $\hS$, then
$\proj_{\hS}(\Apartment_B) = 0$.
\end{compactitem}
\end{lemma}
\begin{proof}
For $1 \leq i \leq k$, let $n_i = \dim_{\Field_p}(S_i)$, so $n_1 + \cdots + n_k = 2g$.  Also, for $0 \leq i \leq k$ define
$T_i = S_1 \oplus \cdots \oplus S_i$, so
\[0 = T_0 \subsetneq T_1 \subsetneq \cdots \subsetneq T_k = \Field_p^{2g}.\]
Let $X$ be the full subcomplex of the Tits building $\Tits_{2g}(\Field_p)$ spanned by vertices
$V$ (i.e.\ nonzero proper subspaces of $\Field_p^{2g}$) such that there exists some $0 \leq i < k$ with
$T_i \subsetneq V \subsetneq T_{i+1}$.  Such a $V$ can be uniquely written as $V = T_i \boxplus W$ for
some nonzero proper subspace $W$ of $S_{i+1}$; the map $W \mapsto T_i \boxplus W$ identifies $\Tits(S_i)$ with
a subcomplex of $X$.  The simplicial complex $X$ is the join of these subcomplexes, so topologically
$X$ is homeomorphic to the join $\Tits(S_1) \ast \cdots \ast \Tits(S_k)$ and thus is a $(2g-k-1)$-dimensional
simplicial complex which is $(2g-k-2)$-connected and satisfies
\[\RH_{2g-k-1}(X;\Q) \cong \St(S_1) \otimes \cdots \otimes \St(S_k) \cong \Sym(\St(S_1),\ldots,\St(S_k)) = \St(S).\]
Our goal is thus to construct a homomorphism
\begin{equation}
\label{eqn:projectiongoal}
\proj_{\hS}\colon \RH_{2g-2}(\Tits_{2g}(\Field_p);\Q) \longrightarrow \RH_{2g-k-1}(X;\Q).
\end{equation}
satisfying the indicated properties.

Before we do this, we make an observation concerning $X$.  Consider a $q$-simplex $\mu$ of $X$.  By definition,
$\mu$ is a flag
\[0 \subsetneq V_0 \subsetneq V_1 \subsetneq \cdots \subsetneq V_{q} \subsetneq \Field_p^{2g}\]
such that each $V_i$ satisfies 
\[T_{j_i} \subsetneq V_i \subsetneq T_{j_i+1}\]
for some $0 \leq j_i \leq k-1$.  There is thus a unique way of
inserting $T_1,\ldots,T_{k-1}$ into $\eta$ in such a way that it remains a flag.  The result is a
$(q+k-1)$-simplex of $\Tits_{2g}(\Field_p)$.  This establishes a bijection
between the $q$-simplices of $X$ and the $(q+k-1)$-simplices of $\Tits_{2g}(\Field_p)$ whose associated
flags contain $T_1,\ldots,T_{k-1}$.

We now turn to constructing \eqref{eqn:projectiongoal}.  For a simplicial complex 
$Z$, denote by $\RC_{\bullet}(Z;\Q)$ the augmented chain complex that computes
$\RH_{\bullet}(Z;\Q)$, so $\RC_{-1}(Z;\Q) = \Q$.  
We first construct a chain-level homomorphism
\[\chainproj_{\hS}\colon \RC_{2g-2}(\Tits_{2g}(\Field_p);\Q) \longrightarrow \RC_{2g-k-1}(X;\Q).\]
It is enough to define $\chainproj_{\hS}(\eta)$ for a $(2g-2)$-dimensional simplex $\eta$ of $\Tits_{2g}(\Field_p)$, which
we do as follows.
\begin{compactitem}
\item If the flag associated to $\eta$ contains $T_1,\ldots,T_{k-1}$, then define $\chainproj_{\hS}(\eta)$ to be
the $(2g-k-1)$-dimensional simplex of $X$ obtained by
deleting $T_1,\ldots,T_{k-1}$ from $\mu$.
\item Otherwise, define $\chainproj_{\hS}(\mu) = 0$.
\end{compactitem}
Next, we prove that $\chainproj_{\hS}$ induces a map on homology.  We remark that this is not obvious -- though
one could define $\chainproj_{\hS}$ on lower-dimensional simplices of $\Tits_{2g}(\Field_p)$ in an analogous manner,
the result is not a map of chain complexes.  Since $\Tits_{2g}(\Field_p)$ and $X$ are
$(2g-2)$-dimensional and $(2g-k-1)$-dimensional, respectively, we have
$\RH_{2g-2}(\Tits_{2g}(\Field_p);\Q) = \RZ_{2g-2}(\Tits_{2g}(\Field_p);\Q)$ and
$\RH_{2g-k-1}(X;\Q) = \RZ_{2g-k-1}(X;\Q)$.  Consider $z \in \RZ_{2g-2}(\Tits_{2g}(\Field_p);\Q)$.  We must
show that $\chainproj_{\hS}(z) \in \RZ_{2g-k-1}(X;\Q)$.  

This is trivial if $g=1$ (in which case $k=1$ and $\chainproj_{\hS}$ is the identity map),
so assume that
$g \geq 2$.  Letting $\zeta$ be a $(2g-k-2)$-simplex of $X$ and letting
$\partial_{\zeta}\colon \RC_{2g-k-1}(X;\Q) \rightarrow \Q$ be the homomorphism that outputs the coefficient
of $\zeta$ under the simplicial boundary map, this is equivalent to showing that $\partial_{\zeta}(\chainproj_{\hS}(z)) = 0$.
Let $\tzeta$ be the $(2g-3)$-simplex of $\Tits_{2g}(\Field_p)$ obtained by inserting $T_1,\ldots,T_{k-1}$ into the flag
associated to $\zeta$ and let $\partial_{\tzeta}\colon \RC_{2g-2}(\Tits_{2g}(\Field_p);\Q) \rightarrow \Q$ be the
homomorphism that outputs the coefficient of $\tzeta$ under the simplicial boundary map.  It is not
necessarily true that $\partial_{\zeta} \circ \chainproj_{\hS} = \partial_{\tzeta}$; however, this is true up to a sign that
depends on $\zeta$ (the $T_i$ terms that are inserted change the signs in the boundary map).  Since
$z \in \RZ_{2g-2}(\Tits_{2g}(\Field_p);\Q)$, we have $\partial_{\tzeta}(z) = 0$, which implies that
$\partial_{\zeta}(\chainproj_{\hS}(z)) = 0$, as desired.

This completes the construction of $\proj_{\hS}$.  We now verify the two indicated properties.  Let
$B = (\vec{v}_1,\ldots,\vec{v}_{2g})$ be an ordered sequence of $2g$ nonzero vectors in $\Field_p^{2g}$.  The
desired formulas are trivial if $\Apartment_B = 0$, so we can assume that the $\vec{v}_i$ form a basis for 
$\Field_p^{2g}$.  For $m \geq 1$, let $\cSym(m)$ denote the symmetric group on $m$ letters.
For $\sigma \in \cSym(2g)$, let $F_{\sigma}$ denote the
element of $\RC_{2g-2}(\Tits_{2g}(\Field_p);\Q)$ corresponding to the flag
\begin{equation}
\label{eqn:expandflag}
0 \subsetneq \Span{\vec{v}_{\sigma(1)}} \subsetneq \Span{\vec{v}_{\sigma(1)},\vec{v}_{\sigma(2)}} \subsetneq \cdots \subsetneq \Span{\vec{v}_{\sigma(1)},\vec{v}_{\sigma(2)},\ldots,\vec{v}_{\sigma(2g)}} = \Field_p^{2g}.
\end{equation}
By definition, the apartment $\Apartment_B$ equals $(2g-2)$-cycle 
\begin{equation}
\label{eqn:expandapartment}
\sum_{\sigma \in \cSym(2g)} (-1)^{\sigma} F_{\sigma}.
\end{equation}
in $\Tits_{2g}(\Field_p)$.
Using this formula, we verify the two indicated properties as follows.
\begin{compactitem}
\item Assume first that $B$ is weakly compatible with $\hS$ and $(\oB_1,\ldots,\oB_k)$ is its projection to $\hS$.
For $\sigma \in \cSym(2g)$, the value $\chainproj_{\hS}(F_{\sigma})$ is nonzero precisely when each $T_i$
appears in \eqref{eqn:expandflag}.  Since $B$ is weakly compatible with $\hS$, this holds for $\sigma = \id$, and
from this we see that this holds precisely for $\sigma \in \cSym(n_1) \times \cdots \times \cSym(n_k) \subset \cSym(2g)$.
It follows that
\begin{align*}
\proj_{\hS}\left(\sum_{\sigma \in \cSym\left(2g\right)} \left(-1\right)^{\sigma} F_{\sigma}\right) &= \sum_{\sigma \in \cSym\left(n_1\right) \times \cdots \times \cSym\left(n_k\right)} (-1)^{\sigma} \chainproj_{\hS}\left(F_{\sigma}\right) \\
&= \Sym\left(\Apartment_{\oB_1},\ldots,\Apartment_{\oB_k}\right) \in \RH_{2g-k-1}\left(X;\Q\right),
\end{align*}
as desired.
\item If no reordering of the vectors in $B$ is weakly compatible with $\hS$, then $\chainproj_{\hS}$
takes every term of \eqref{eqn:expandapartment} to $0$, so $\proj_{\hS}(\Apartment_B) = 0$.\qedhere
\end{compactitem}
\end{proof}

We now prove Lemma \ref{lemma:projectionmapunordered}.

\begin{proof}[Proof of Lemma \ref{lemma:projectionmapunordered}]
We first recall the statement.  Let $p$ be a prime and let $g \geq 1$.  Let
$S = \{S_1,\ldots,S_k\}$ be a symplectic splitting of $\Field_p^{2g}$.  Our goal is to construct
a linear map $\pi_S\colon \St_{2g}(\Field_p) \rightarrow \St(S)$ such that if
$B_i$ is an ordered sequence of $\dim(S_i)$ nonzero vectors in $S_i$ for $1 \leq i \leq k$, then
$\pi_S(\Apartment_{B_1 \cdots B_k}) = \Sym(\Apartment_{B_1},\ldots,\Apartment_{B_k})$.

Fix an ordering
$\hS = (S_1,\ldots,S_k)$ of $S$.  Let $\cSym(k)$ be the symmetric group on $k$ letters.  For $\sigma \in \cSym(k)$,
define $\sigma(\hS) = (S_{\sigma(1)},\ldots,S_{\sigma(k)})$ and let $\pi_{\sigma(\hS)}$ be the map
given by Lemma \ref{lemma:projectionmapordered}.  We then define
$\pi_S\colon \St_{2g}(\Field_p) \rightarrow \St(S)$ via the formula
\begin{equation}
\label{eqn:projectionmap}
\pi_S(x) = \frac{1}{k!} \sum_{\sigma \in \cSym(k)} \pi_{\sigma(\hS)}(x) \quad \quad (x \in \St_{2g}(\Field_p)).
\end{equation}
This does not depend on the choice of ordering $\hS$.  

We now verify the desired formula.  For $1 \leq i \leq k$, let $B_i$ be an ordered sequence of $\dim(S_i)$
nonzero vectors in $S_i$.  Since $\dim(S_i)$ is even for all $i$, we have
\[\Apartment_{B_{\sigma(1)} \cdots B_{\sigma(k)}} = \Apartment_{B_1 \cdots B_k} \quad \quad (\sigma \in \cSym(k)).\]
It follows that
\begin{align*}
\pi_S(\Apartment_{B_1 \cdots B_k}) &= \frac{1}{k!} \sum_{\sigma \in \cSym(k)} \pi_{\sigma(\hS)}(\Apartment_{B_1 \cdots B_k}) \\
&= \frac{1}{k!} \sum_{\sigma \in \cSym(k)} \pi_{\sigma(\hS)}(\Apartment_{B_{\sigma(1)} \cdots B_{\sigma(k)}}) \\
&= \frac{1}{k!} \sum_{\sigma \in \cSym(k)} \Sym(\Apartment_{B_{\sigma(1)}},\ldots,\Apartment_{B_{\sigma(k)}}) \\
&= \Sym(\Apartment_{B_1},\ldots,\Apartment_{B_k}),
\end{align*}
as desired.
\end{proof}

\subsection{Cross-projections: the proof of Lemma \ref{lemma:decomposecrossproj}}
\label{section:decomposecrossproj}

This section is devoted to the proof of Lemma \ref{lemma:decomposecrossproj}.

\begin{proof}[Proof of Lemma \ref{lemma:decomposecrossproj}]
We first recall the statement.  Let $p$ be a prime and let $g \geq 1$.  
Let $S,S' \in \fS_g$ be such that $S \npreceq S'$.  Letting $\pi_S\colon \St_{2g}(\Field_p) \rightarrow \St(S)$
be the projection map from Lemma \ref{lemma:projectionmapunordered}, we must prove that
$\pi_S(\bV_g(S')) \subset \StSep(S)$.  Recall from \eqref{eqn:projectionmap} that $\pi_S$ is a
linear combination of the maps $\pi_{\hS}$ provided by Lemma \ref{lemma:projectionmapordered}, where
$\hS$ ranges over the various orderings of $S$.  Fixing some ordering $\hS = (S_1,\ldots,S_k)$ of
$S$, it is thus enough prove that $\pi_{\hS}(\bV_g(S')) \subset \StSep(S)$.  

Write $S' = \{S_1',\ldots,S_{\ell}'\}$.  Keeping in mind that reordering the vectors making up an apartment only
changes the apartment by a sign, $\bV_g(S')$ is generated by elements $\Apartment_B$, where $B$ is an
ordered basis for $\Field_p^{2g}$ such that every vector in $B$ lies in some $S_i'$.  For such an
$\Apartment_B$, we must show that $\pi_{\hS}(\Apartment_B) \in \StSep(S)$.  If $B$ cannot be reordered
so as to be weakly compatible with $\hS$, then Lemma \ref{lemma:projectionmapordered} says that
$\pi_{\hS}(\Apartment_B) = 0$, proving the lemma in this case.  We can thus assume that $B$ can be
reordered so as to be weakly compatible with $\hS$.  Reorder it in this way (changing it by a sign), and
let $B = B_1 \cdots B_k$ be its factorization and $(\oB_1,\ldots,\oB_k)$ be its projection.  By definition,
for $1 \leq i \leq k$ this means that $B_i \subset S_1 \boxplus \cdots \boxplus S_i$ (where recall that
$\boxplus$ denotes the internal direct sum) and that $\oB_i$
is the image of $B_i$ under the orthogonal projection $\Field_p^{2g} \rightarrow S_i$.
Lemma
\ref{lemma:projectionmapordered} says that 
\[\pi_{\hS}(\Apartment_B) = \Sym(\Apartment_{\oB_1},\ldots,\Apartment_{\oB_k}) \in \St(S).\]
We must prove that 
\begin{equation}
\label{eqn:crossprojtoprove}
\Sym(\Apartment_{\oB_1},\ldots,\Apartment_{\oB_k}) \in \StSep(S).
\end{equation}
We will do this by using $S$ and $S'$ to construct a symplectic splitting $\oT$ that refines both
$S$ and $S'$ such that
\[\Sym(\Apartment_{\oB_1},\ldots,\Apartment_{\oB_k}) \in \St(\oT);\]
since $S'$ is not a refinement of $S$, this will imply \eqref{eqn:crossprojtoprove}.

Since permuting the vectors in an apartment class only changes the apartment class by a sign,
we can permute the vectors in each $B_i$ freely.  Doing this, we can assume without loss of
generality that each $B_i$ is of the form $B_i = B_{i1} \cdots B_{i\ell}$, where $B_{ij}$ is a sequence
of nonzero vectors (possibly the empty sequence) lying in $S_j'$.  For $1 \leq i \leq k$ and
$1 \leq j \leq \ell$, let $\oB_{ij}$ be the image of $B_{ij}$ in $\oB_i$, so 
$\oB_i = \oB_{i1} \cdots \oB_{i\ell}$.

For $1 \leq i \leq k$ and $1 \leq j \leq \ell$, let $T_{ij}$ be the span of the vectors in $B_{ij}$ and
let $\oT_{ij}$ be the span of the vectors in $\oB_{ij}$.  By construction, we have
\[T_{ij} \subset S_1 \boxplus \cdots \boxplus S_i \quad \text{and} \quad T_{ij} \subset S_j' \quad \text{and} \quad \oT_{ij} \subset S_i.\]
One might worry that in orthogonally projecting $T_{ij}$ into $S_i$ to form $\oT_{ij}$ we might have destroyed the fact
that it lies in $S_j'$; however, the following claim shows that this has not happened.
\begin{claim}
For $1 \leq i \leq k$ and $1 \leq j \leq \ell$, we have $\oT_{ij} \subset S_j'$.
\end{claim}
\begin{proof}[Proof of claim]
Consider $x \in T_{ij}$.  Define $S_{1,i-1} = S_1 \boxplus \cdots \boxplus S_{i-1}$, so
$S_{1,i-1}$ and $S_i$ are orthogonal symplectic subspaces of $\Field_p^{2g}$ satisfying
\[x \in T_{ij} \subset S_{1,i-1} \boxplus S_i.\]
Let $\ox_1 \in S_{1,i-1}$ and $\ox_2 \in S_i$ be the orthogonal
projections of $x$ to these subspaces, so $x = \ox_1 + \ox_2$.  Our goal is to prove that $\ox_2 \in S_j'$.
Since $x \in S_j'$, this is equivalent to showing that $\ox_1 \in S_j'$.

For $1 \leq j' \leq \ell$, let $U_{j'}$ be the span of the $T_{i' j'}$ for $1 \leq i' \leq i-1$.  We thus have
\[U_{1} \boxplus U_{2} \boxplus \cdots \boxplus U_{\ell} = S_{1,i-1} \quad \text{and} \quad U_{j'} \subset S_{j'}'.\]
Since $U_{j'} \subset S_{j'}'$ and the $S_{j'}'$ are pairwise orthogonal, so are the $U_{j'}$.  Since
the $U_{j'}$ are pairwise orthogonal subspaces of the symplectic subspace $S_{1,i-1}$ that span
$S_{1,i-1}$, Lemma \ref{lemma:orthogonalsymplectic} implies that the $U_{j'}$ themselves are symplectic subspaces.
For $1 \leq j' \leq \ell$, let $\ox_{1,j'}$ be the orthogonal projection of $x$ to $U_{j'}$.  We
then have $\ox_1 = \ox_{1,1} + \ox_{1,2} + \cdots + \ox_{1,\ell}$. 
For $1 \leq j' \leq \ell$ with $j' \neq j$, the subspace $S_j'$ is orthogonal to $U_{j'}$, so the
orthogonal projection of $S_j'$ to $U_{j'}$ is $0$.  Since $x \in S_j'$, this implies that $\ox_{1,j'} = 0$
for $1 \leq j' \leq \ell$ with $j' \neq j$.  We conclude that $\ox_1 = \ox_{1,j} \in S_j'$, as desired.
\end{proof}

Define
\[\oT = \Set{$\oT_{ij}$}{$1 \leq i \leq k$, $1 \leq j \leq \ell$, $\oT_{ij} \neq 0$}.\]
Since $\oT_{ij} \subset S_i \cap S_j'$ and both $S$ and $S'$ are symplectic splittings, it follows that
the elements of $\oT$ are pairwise orthogonal to each other.  They also span $\Field_p^{2g}$, so Lemma \ref{lemma:orthogonalsymplectic}
implies that $\oT$ is a symplectic splitting of $\Field_p^{2g}$.  We have $\oT \preceq S$ and $\oT \preceq S'$; since
$S$ is not a refinement of $S'$, the symplectic splitting $\oT$ must be a proper refinement of $S$.
This implies that there exists some
$1 \leq i \leq k$ such that
\[\Set{$\oT_{ij}$}{$1 \leq j \leq \ell$, $\oT_{ij} \neq 0$} \neq \{S_i\}.\]
Since $\oT_{ij}$ is the span of the vectors in $\oB_{ij}$ and since
$\oB_i = \oB_{i1} \cdots \oB_{i\ell}$, we deduce that $\Apartment_{\oB_i} \in \StSep(S_i)$ and thus
that
\[\Sym(\Apartment_{\oB_1},\ldots,\Apartment_{\oB_{k}}) \in \StSep(S),\]
as desired.
\end{proof}

\section{Bounds}
\label{section:bounds}

In this section, we prove Proposition \ref{proposition:nosepbig}.  We first recall its statement.  Fix a prime $p$.  For
$g \geq 1$, define
\[\theta_g = \frac{\dim_{\Q} \StNS_{2g}(\Field_p)}{|\Sp_{2g}(\Field_p)|} \quad \text{and} \quad \lambda_g = \frac{1}{g(p^{2g}-1)}.\]
Proposition \ref{proposition:nosepbig} asserts that $\theta_g = \lambda_g$ for all $g \geq 1$.

Before we prove this, we must prove two lemmas that assert that the $\theta_g$ and $\lambda_g$ satisfy similar
recurrence relations.  
Given $n \geq 1$, a {\em partition} of $n$ is an expression
\begin{equation}
\label{eqn:partition}
n = n_1 + n_2 + \cdots + n_m
\end{equation}
with 
\[n_1 \geq n_2 \geq \cdots \geq n_m \geq 1.\]
If the numbers that appear in \eqref{eqn:partition} are $a_1 > a_2 > \cdots > a_\ell$ and $a_i$ appears
$r_i$ times, then we will denote the partition by $(a_1^{r_1},\ldots,a_{\ell}^{r_{\ell}})$ and say as shorthand
that
\[(a_1^{r_1},\ldots,a_{\ell}^{r_{\ell}}) \vdash n.\]
Our first lemma is then as follows.  The proof of this lemma is where we use Theorem \ref{theorem:decomposition}.

\begin{lemma}
\label{lemma:thetarecurrence}
Let $g \geq 1$ and let $p$ be a prime.  Then
\[\frac{p^{\binom{2g}{2}}}{|\Sp_{2g}(\Field_p)|} = \sum_{(a_1^{r_1},\ldots,a_{\ell}^{r_{\ell}}) \vdash g} \frac{(\theta_{a_1})^{r_1} (\theta_{a_2})^{r_2} \cdots (\theta_{a_{\ell}})^{r_{\ell}}}{(r_1)! (r_2)! \cdots (r_{\ell})!}.\]
\end{lemma}
\begin{proof}
We begin with some combinatorial facts about symplectic splittings.
Given an ordered symplectic splitting $\hS = (S_1,\ldots,S_k)$ of $\Field_p^{2g}$ and a partition 
$(a_1^{r_1},\ldots,a_{\ell}^{r_{\ell}}) \vdash g$, 
we say that the {\em type} of $\hS$ is 
$(a_1^{r_1},\ldots,a_{\ell}^{r_{\ell}})$ 
if the sequence of numbers 
\[(\dim_{\Field_p}(S_1), \dim_{\Field_p}(S_2),\ldots,\dim_{\Field_p}(S_k))\]
begins with $r_1$ entries of $2a_1$, then has $r_2$ entries
of $2a_2$, etc.
We claim that there are
\begin{equation}
\label{eqn:numordered}
\frac{|\Sp_{2g}(\Field_p)|}{\prod_{i=1}^{\ell} |\Sp_{2a_i}(\Field_p)|^{r_i}}
\end{equation}
ordered symplectic splittings of type $(a_1^{r_1},\ldots,a_{\ell}^{r_{\ell}})$.  Indeed, the
group $\Sp_{2g}(\Field_p)$ acts transitively on the ordered symplectic splittings of type
$(a_1^{r_1},\ldots,a_{\ell}^{r_{\ell}})$ and
$\prod_{i=1}^{\ell} \Sp_{2a_i}(\Field_p)^{r_i}$
is the stabilizer of one such splitting.

Next, we say that an (unordered) symplectic splitting
has type $(a_1^{r_1},\ldots,a_{\ell}^{r_{\ell}})$ if some ordering of it has that type.
We claim that there are
\begin{equation}
\label{eqn:numunordered}
\frac{|\Sp_{2g}(\Field_p)|}{\prod_{i=1}^{\ell} (r_i)! |\Sp_{2a_i}(\Field_p)|^{r_i}}
\end{equation}
symplectic splittings of type $(a_1^{r_1},\ldots,a_{\ell}^{r_{\ell}})$.  Indeed, there are
$(r_1)! (r_2)! \cdots (r_{\ell})!$ orderings of that type for each unordered splitting, so the
equation follows from \eqref{eqn:numordered} above.

We now turn to the desired identity.  Theorem \ref{theorem:decomposition} implies that
\begin{equation}
\label{eqn:decomp1}
\dim_{\Q} \St_{2g}(\Field_p) = \sum_{S \in \fS_g} \dim_{\Q} \StNS(S).
\end{equation}
The Solomon--Tits theorem \cite[Theorem IV.5.2]{BrownBuildings} implies that 
$\dim_{\Q} \St_{2g}(\Field_p) = p^{\binom{2g}{2}}$.
Also, by definition for $S \in \fS_g$ with $S = \{S_1,\ldots,S_k\}$ we have
$\dim_{\Q} \StNS(S) = \prod_{i=1}^k \StNS(S_i)$.
Plugging these into \eqref{eqn:decomp1} and grouping the elements of $\fS_g$ together by type, we can
use \eqref{eqn:numunordered} to deduce that
\begin{align*}
p^{\binom{2g}{2}} &= \sum_{(a_1^{r_1},\ldots,a_{\ell}^{r_{\ell}}) \vdash g} \frac{|\Sp_{2g}(\Field_p)| (\dim_{\Q} \StNS_{2a_1})^{r_1} (\dim_{\Q} \StNS_{2a_2})^{r_2} \cdots (\dim_{\Q} \StNS_{2a_{\ell}})^{r_{\ell}}}{\prod_{i=1}^{\ell} (r_i)! |\Sp_{2a_i}(\Field_p)|^{r_i}} \\
&= \sum_{(a_1^{r_1},\ldots,a_{\ell}^{r_{\ell}}) \vdash g} \frac{|\Sp_{2g}(\Field_p)| (\theta_{a_1})^{r_1} (\theta_{a_2})^{r_2} \cdots (\theta_{a_{\ell}})^{r_{\ell}}}{(r_1)! (r_2)! \cdots (r_{\ell})!}.
\end{align*}
Dividing both sides by $|\Sp_{2g}(\Field_p)|$ gives the desired result.
\end{proof}

Our second lemma is a similar identity for $\lambda_g = \frac{1}{g(p^{2g}-1)}$.
Its proof was suggested to the authors by
the MathOverflow user ``Lucia''; see \cite{MathOverflow}.

\begin{lemma}
\label{lemma:lambdarecurrence}
Let $g \geq 1$ and let $p$ be a prime.  Then
\[\frac{p^{\binom{2g}{2}}}{|\Sp_{2g}(\Field_p)|} = \sum_{(a_1^{r_1},\ldots,a_{\ell}^{r_{\ell}}) \vdash g} \frac{(\lambda_{a_1})^{r_1} (\lambda_{a_2})^{r_2} \cdots (\lambda_{a_{\ell}})^{r_{\ell}}}{(r_1)! (r_2)! \cdots (r_{\ell})!}.\]
\end{lemma}
\begin{proof}
Using the standard formula for $|\Sp_{2g}(\Field_p)|$ (see, e.g., \cite[p. 27]{GroveBook}), the left hand side of the purported identity equals
\begin{align*}
\frac{p^{g(2g-1)}}{(p^{2g}-1)p^{2g-1}(p^{2g-2}-1)p^{2g-3} \cdots (p^2-1)p}
&= \frac{p^{g(2g-1)}}{p^{g^2}(p^{2g}-1)(p^{2g-2}-1)\cdots(p^{2}-1)}\\
&= \frac{p^{g(g-1)}}{(p^{2g}-1)(p^{2g-2}-1)\cdots(p^{2}-1)}.
\end{align*}
We will show that this equals the right hand side using generating functions.
The Exponential Formula \cite[Corollary 5.1.9]{StanleyBookII} gives a formal
power series identity
\begin{equation}
\label{eqn:oneside}
\exp\left(\sum_{g=1}^{\infty} \lambda_g x^g\right) = 1+ \sum_{g=1}^{\infty} \left(\sum_{(a_1^{r_1},\ldots,a_{\ell}^{r_{\ell}}) \vdash g} \frac{(\lambda_{a_1})^{r_1} (\lambda_{a_2})^{r_2} \cdots (\lambda_{a_{\ell}})^{r_{\ell}}}{(r_1)! (r_2)! \cdots (r_{\ell})!}\right) x^g.
\end{equation}
Plugging in our formula $\lambda_g = \frac{1}{g(p^{2g}-1)}$, we have
\begin{align*}
\exp\left(\sum_{g=1}^{\infty} \lambda_g x^g\right) &=
\exp\left(\sum_{g=1}^{\infty}  \frac{x^g}{g(p^{2g}-1)}\right) 
= \exp\left(-\sum_{g=1}^{\infty} \left(\frac{x^g}{g} \sum_{h=0}^{\infty} p^{2gh}\right) \right) \\
&= \exp\left(-\sum_{h=0}^{\infty} \sum_{g=1}^{\infty} \frac{(x p^{2h})^g}{g} \right) 
= \exp\left(\sum_{h=0}^{\infty}\log(1-p^{2h} x)\right) \\
&= \prod_{h=0}^{\infty} (1-p^{2h} x).
\end{align*}
A theorem of Euler (see \cite[Corollary 2.2]{AndrewsBook}) asserts that
\begin{align}
\prod_{h=0}^{\infty} (1-p^{2h} x) &= 1 + \sum_{g=1}^{\infty} \left(\frac{(-1)^g p^{g(g-1)}}{(1-p^{2g})(1-p^{2g-2})\cdots(1-p^2)}\right) x^g \nonumber\\
&= 1 + \sum_{g=1}^{\infty} \left(\frac{p^{g(g-1)}}{(p^{2g}-1)(p^{2g-2}-1)\cdots(p^2-1)}\right) x^g. \label{eqn:otherside}
\end{align}
Comparing the coefficients of $x^g$ in \eqref{eqn:oneside} and \eqref{eqn:otherside} gives the desired identity.
\end{proof}

\begin{proof}[Proof of Proposition \ref{proposition:nosepbig}]
We will prove that $\theta_g = \lambda_g$ for $g \geq 1$ by induction on $g$.  The base case $g=1$ asserts that
\[\frac{\dim_{\Q} \StNS_{2}(\Field_p)}{|\Sp_{2}(\Field_p)|} = \frac{1}{(p^{2}-1)}.\]
Since there are no nontrivial symplectic splittings of $\Field_p^2$, we have $\StNS_2(\Field_p) = \St_2(\Field_p)$.
The Solomon--Tits Theorem \cite[Theorem IV.5.2]{BrownBuildings} says that $\St_2(\Field_p)$ is $p$-dimensional.  What
is more, $|\Sp_{2}(\Field_p)| = (p^2-1)p$ (see, e.g., \cite[p. 27]{GroveBook}).  The above identity follows.

Now assume that $g>1$ and that $\theta_{g'} = \lambda_{g'}$ for all $1 \leq g' < g$.  Applying
Lemmas \ref{lemma:thetarecurrence} and \ref{lemma:lambdarecurrence} together with our inductive hypothesis, we see that
\begin{align*}
\theta_g &= \frac{p^{\binom{2g}{2}}}{|\Sp_{2g}(\Field_p)|} - \sum_{\substack{(a_1^{r_1},\ldots,a_{\ell}^{r_{\ell}}) \vdash g \\ \text{except $(g) \vdash g$}}} \frac{(\theta_{a_1})^{r_1} (\theta_{a_2})^{r_2} \cdots (\theta_{a_{\ell}})^{r_{\ell}}}{(r_1)! (r_2)! \cdots (r_{\ell})!}\\
&= \frac{p^{\binom{2g}{2}}}{|\Sp_{2g}(\Field_p)|} - \sum_{\substack{(a_1^{r_1},\ldots,a_{\ell}^{r_{\ell}}) \vdash g \\ \text{except $(g) \vdash g$}}} \frac{(\lambda_{a_1})^{r_1} (\lambda_{a_2})^{r_2} \cdots (\lambda_{a_{\ell}})^{r_{\ell}}}{(r_1)! (r_2)! \cdots (r_{\ell})!}
= \lambda_g. \qedhere
\end{align*}
\end{proof}

\noindent
\begin{tabular*}{\linewidth}[t]{@{}p{\widthof{E-mail: {\tt neil.fullarton@rice.edu}}+1in}@{}p{\linewidth - \widthof{E-mail: {\tt neil.fullarton@rice.edu}} - 1in}@{}}
{\raggedright
Neil Fullarton\\
Department of Mathematics\\
Rice University, MS 136 \\
6100 Main St.\\
Houston, TX 77005\\
E-mail: {\tt neil.fullarton@rice.edu}}
&
{\raggedright
Andrew Putman\\
Department of Mathematics\\
University of Notre Dame\\
279 Hurley Hall\\ 
Notre Dame, IN 46556\\
E-mail: {\tt andyp@nd.edu}}
\end{tabular*}

\end{document}